\documentclass{amsart}

\usepackage[T1]{fontenc}
\usepackage[utf8]{inputenc}

\usepackage{color}
\usepackage[colorlinks=true]{hyperref}
\hypersetup{urlcolor=blue, citecolor=blue}

\usepackage{amsmath,amssymb,amsfonts,amsthm} 
\usepackage{graphicx}
\usepackage{pgfplots}
\usepackage{cleveref}
\usepackage{thmtools}
\usepackage{stmaryrd}
\usepackage{mathtools}
\usepackage{mathrsfs}  
\usepackage{enumerate}
\usepackage{dsfont}

\usepackage[normalem]{ulem} 
\usepackage{cancel}
\usepackage[color=red!65!yellow]{todonotes}

\usepackage{subfig}
\usepackage{booktabs}
\usepackage{memhfixc}
 \usepackage{dsfont}
\usepackage{psfrag}
\usepackage{picture}

\usepackage{lineno}

\newtheorem{thm}{Theorem}[section]
\newtheorem{lem}[thm]{Lemma}
\newtheorem{prop}[thm]{Proposition}

\newtheorem{theorem}{Theorem}[section]
\newtheorem{lemma}[theorem]{Lemma}

\newtheorem{proposition}[theorem]{Proposition}

\newtheorem{remark}[theorem]{Remark}

\theoremstyle{definition}
\newtheorem{definition}[theorem]{Definition}

\newcommand{\cA}{\ensuremath{\mathcal{A}}}

\newcommand{\cC}{\ensuremath{\mathcal{C}}}

\newcommand{\cE}{\ensuremath{\mathcal{E}}}

\newcommand{\cH}{\ensuremath{\mathcal{H}}}

\newcommand{\cT}{\ensuremath{\mathcal{T}}}

\newcommand{\cV}{\ensuremath{\mathcal{V}}}

\newcommand{\bN}{\ensuremath{\mathbb{N}}}

\newcommand{\bR}{\ensuremath{\mathbb{R}}}

\newcommand{\bbu}{\boldsymbol{u}}
\newcommand{\bbv}{\boldsymbol{v}}
\newcommand{\bbm}{\boldsymbol{m}}
\newcommand{\bbw}{\boldsymbol{w}}

\newcommand{\bbF}{\boldsymbol{F}}
\newcommand{\bbG}{\boldsymbol{G}}

\newcommand{\bbJ}{\boldsymbol{J}}
\newcommand{\bbphi}{\boldsymbol{\phi}}

\newcommand{\bbun}{\boldsymbol{1}}

\newcommand{\un}{\mathds 1}
\def\1{\boldsymbol{1}}

\newcommand{\logLogSlopeTriangle}[5]
{
	\pgfplotsextra
	{
		\pgfkeysgetvalue{/pgfplots/xmin}{\xmin}
		\pgfkeysgetvalue{/pgfplots/xmax}{\xmax}
		\pgfkeysgetvalue{/pgfplots/ymin}{\ymin}
		\pgfkeysgetvalue{/pgfplots/ymax}{\ymax}
		
		\pgfmathsetmacro{\xArel}{#1}
		\pgfmathsetmacro{\yArel}{#3}
		\pgfmathsetmacro{\xBrel}{#1-#2}
		\pgfmathsetmacro{\yBrel}{\yArel}
		\pgfmathsetmacro{\xCrel}{\xArel}
		
		\pgfmathsetmacro{\lnxB}{\xmin*(1-(#1-#2))+\xmax*(#1-#2)} 
		\pgfmathsetmacro{\lnxA}{\xmin*(1-#1)+\xmax*#1} 
		\pgfmathsetmacro{\lnyA}{\ymin*(1-#3)+\ymax*#3} 
		\pgfmathsetmacro{\lnyC}{\lnyA-#4*(\lnxA-\lnxB)}
		\pgfmathsetmacro{\yCrel}{\lnyC-\ymin)/(\ymax-\ymin)}
		
		\coordinate (A) at (rel axis cs:\xArel,\yArel);
		\coordinate (B) at (rel axis cs:\xBrel,\yBrel);
		\coordinate (C) at (rel axis cs:\xCrel,\yCrel);
		
		\draw[#5]   (A)-- node[pos=0.5,anchor=north] {\scriptsize{1}}
		(B)--
		(C)-- node[pos=0.,anchor=east] {\scriptsize{#4}} 
		(A);
	}
}

\def\clem#1{#1}
\def\laurent#1{#1}

\begin{document}
\title{Finite Volumes for the Stefan-Maxwell cross-diffusion system}
\author{Cl\'ement Canc\`es} 
\address{Cl\'ement  \textsc{Canc\`es} (\href{clement.cances@inria.fr}{\tt clement.cances@inria.fr})\\
Inria, Univ. Lille, CNRS, UMR 8524 - Laboratoire Paul Painlev\'e, F-59000 Lille.}
\author{Virginie Ehrlacher} 
\address{Virginie \textsc{Ehrlacher} (\href{virginie.ehrlacher@enpc.fr}{\tt virginie.ehrlacher@enpc.fr})\\
CERMICS, Ecole des Ponts ParisTech and Inria Paris, Université Paris-Est, 6-8 avenue Blaise Pascal, 77455, Marne-la-Vallée, France.
}
\author{Laurent Monasse}
\address{Laurent \textsc{Monasse} (\href{laurent.monasse@inria.fr}{\tt laurent.monasse@inria.fr})\\
Universit\'e C\^ote d’Azur, Inria, CNRS, Laboratoire J.A. Dieudonn\'e, Team Coffee, Parc Valrose, 06108 Nice cedex 02, France.
}
\begin{abstract}
The aim of this work is to propose a provably convergent finite volume scheme for the so-called Stefan-Maxwell model, which describes the evolution of the composition of a multi-component mixture and reads as a cross-diffusion system. The scheme proposed here 
relies on a two-point flux approximation, and preserves at the discrete level some fundamental theoretical properties of the continuous models, namely the non-negativity of the solutions, the conservation of mass and the preservation of the volume-filling constraints. 
In addition, the scheme satisfies a discrete entropy-entropy dissipation relation, very close to the relation which holds at the continuous level. In this article, we present this scheme together with its numerical analysis, and finally 
illustrate its behaviour with some numerical results. 
\end{abstract}
\maketitle

\section{The Stefan-Maxwell model}

The aim of this section is to present the so-called Stefan-Maxwell model, which is introduced in Section~\ref{sec:pres}. Its key mathematical properties are summarized in Section~\ref{sec:prop}. In particular, an 
entropy-entropy dissipation inequality holds for this system and is formally derived in Section~\ref{sec:entrop}. 

\subsection{Presentation of the model}\label{sec:pres}

The Maxwell-Stefan equations describe the evolution of the composition of a multicomponent mixture via diffusive transport~\cite{maxwell1867iv,stefan63ueber}. This model is used in various 
applications like sedimentation, dialysis, electrolysis,
ion exchange, ultrafiltration, and respiratory airways~\cite{wesselingh2000mass}.

\normalfont

We are interested in the evolution of the composition of a mixture of $n\in \mathbb{N}^*$ species, which is described by
the volume fractions $u = (u_1, \cdots, u_n)$, where $u_i$ denotes the volume fraction of the $i^{th}$ species for all $1\leq i \leq n$.
The spatial domain occupied by the mixture is represented by an open, connected, bounded, and polyhedral subset $\Omega$ of $\mathbb{R}^d$. Let $T>0$ denote some arbitrary final time. 

For all $1\leq i\neq j \leq n$, 
let $c_{ij} = c_{ji}>0$ be some positive real numbers. The coefficient $c_{ij}$ can be interpreted as the inverse of the inter-species diffusion coefficient between the 
$i^{th}$ and $j^{th}$ species. For all $v:=(v_1,\cdots,v_n) \in \mathbb{R}_+^n$, we denote by $A(v):=(A_{ij}(v))_{1\leq i,j \leq n}$ the matrix defined by
\begin{equation}\label{eq:defA}
A_{ii}(v):= \sum_{1\leq j \neq i \leq n} c_{ij}v_j, \quad A_{ij}(v):= - c_{ij} v_i.
\end{equation}

In the Stefan-Maxwell model, the evolution of the composition of the mixture is prescribed by the following system of partial differential equations:
\begin{equation}\label{eq:masstrans}
\partial_t u_i + {\rm div} J_i = 0, \quad \forall 1\leq i \leq n, 
\end{equation}
where the set of fluxes $J:=(J_i)_{1\leq i \leq n}$ is solution to the set of equations
\begin{align}\label{eq:SM1}
& \nabla u_i + \sum_{j=1}^n A_{ij}(u) J_j = 0, \quad \forall 1\leq i \leq n,\\ \label{eq:SM2}
& \sum_{i=1}^n J_i = 0.
\end{align}
For any vectors $v:=(v_i)_{1\leq i \leq n},w:=(w_i)_{1\leq i \leq n}\in \bR^n$, we denote by $\langle v, w\rangle:= \sum_{i=1}^n v_i w_i$ the canonical scalar product of $v,w$ in $\bR^n$, while the canonical scalar product of vectors $F,G \in \bR^d$ is denoted by $F\cdot G$.
Equations (\ref{eq:SM1}) and (\ref{eq:SM2}) can then be rewritten in the more compact form
\begin{align}\label{eq:SM1comp}
& \nabla u + A(u) J = 0, \\ \label{eq:SM2comp}
& \langle \un, J \rangle= 0,
\end{align}
where $\un:=(1,1,\cdots,1)\in \bR^n$. We refer the reader to Appendix~A of~\cite{jungel2013existence} and \cite{BGS15}
for the derivation of the model (\ref{eq:masstrans})-(\ref{eq:SM1})-(\ref{eq:SM2}).

The system is complemented with no-flux boundary conditions 
\begin{equation}\label{eq:no-flux}
J_i \cdot n = 0 \; \text{on}\; \partial \Omega, \quad \text{for all}\; 1\leq i \leq n,
\end{equation}
and a {measurable}\; initial condition 
$u^0 = (u^0_1, \cdots, u^0_n)$ which satisfies 
\begin{equation}\label{eq:init}
\forall 1\leq i \leq n , \quad u_i^0 \geq 0 \quad \mbox{ and } \quad \sum_{i=1}^n u_i^0= 1 \quad \mbox{ on } \partial \Omega. 
\end{equation}
In other words, denoting by
$$
\mathcal A:= \left\{ v \in \bR_+^n ,\quad \langle \un, v \rangle  = 1\right\}, 
$$
we assume that $u^0 \in L^\infty(\Omega; \mathcal A)$. 
Let us also assume in addition that 
\begin{equation}\label{eq:masspos.init}
\forall 1\leq i \leq n, \quad M_i:=\int_\Omega u_i^0 >0,
\end{equation}
i.e. that each of the different species is initially present in the mixture. We denote by $M = \left(M_i\right)_{1 \leq i \leq n} \in (\bR_+^*)$ 
the vector of masses. Since $u_0 \in  L^\infty(\Omega; \mathcal A)$, one has $\langle \un, M \rangle = m_\Omega$ where 
$m_\Omega$ stands for the Lebesgue measure of $\Omega$.

The mathematical analysis of the Stefan-Maxwell model is quite recent~\cite{giovangigli2012multicomponent,bothe2011maxwell,boudin2012mathematical,jungel2013existence}. 
The first existence result of global weak solutions to the Stefan-Maxwell problem for general initial data and number of chemical species was proved in~\cite{jungel2013existence}. 

Motivated by the results of~\cite{jungel2013existence}, we introduce here the notion of weak solution to the Stefan-Maxwell system of equations, which is used in our analysis. {In what follows, we denote by $Q_T = (0,T) \times \Omega$, and by 
$$
\cV_\lambda =  \left\{ v = (v_1, \dots, v_n) \; \middle| \; \sum_i v_i = \langle \un, v \rangle = \lambda \right\}, \qquad 
\lambda \in \bR.
$$
In particular, $\cA = \cV_1 \cap (\bR_+)^n$, $J \in (\cV_0)^d$, and $M \in \cV_{m_\Omega}$.
}

\begin{definition}\label{def:defsol}
 A weak solution $(u,J)$ to \eqref{eq:masstrans}-\eqref{eq:SM1comp}-\eqref{eq:SM2comp} 
 corresponding to the initial profile $u^0\in L^\infty(\Omega; \cA)$ is a pair $(u,J)$ 
 such that $u \in L^\infty(Q_T; \cA)\cap L^2((0,T); H^1(\Omega))^n$ and $\nabla \sqrt{u} \in L^2(Q_T)^{n\times d}$, 
 such that $J\in {L^2(Q_T;(\cV_0)^d)}$ satisfies~\eqref{eq:SM1comp-2}, 
 and such that, for all $\phi  \in \cC^\infty_c([0,T)\times \overline{\Omega})^n$,
\begin{equation}
 \iint_{Q_T} \langle u, \partial_t \phi\rangle + \int_\Omega \langle u^0 ,\phi(0,\cdot) \rangle 
 + \iint_{Q_T} \sum_{i=1}^n J_i \cdot \nabla \phi_i = 0.
 \label{eq:weak.u}
\end{equation}
\end{definition}

\subsection{Key mathematical properties of the model}\label{sec:prop}

In this section, we exhibit some key mathematical properties of the model, which were proved in~\cite{jungel2013existence}, and that we wish to preserve at the discrete level in the numerical scheme. 

First, the total mass of each specie is conserved, i.e, for all $1\leq i \leq n$ and $t>0$, 
\begin{equation}\label{eq:cont.mass}
\int_{\Omega} u_i(t,x)\,dx = \int_\Omega u_i^0(x)\,dx.
\end{equation}
This follows directly from the local conservation
property (\ref{eq:masstrans}) and the no-flux boundary conditions across $\partial \Omega$. 

Second, the {volume fractions} remain non-negative, i.e., 
\begin{equation}\label{eq:cont.pos}
\forall 1\leq i \leq n, \quad u_i(t,x) \geq 0, \mbox{ for almost all }(t,x)\in Q_T.
\end{equation}

Third, the condition \eqref{eq:SM2} together with \eqref{eq:masstrans} implies that 
$\partial_t \langle \un, u\rangle  = 0,$ 
so that condition~\eqref{eq:init} on the initial condition yields
\begin{equation}\label{eq:cont.sum1}
\sum_{i=1}^n u_i(t,x)=1 \quad \mbox{ for almost all }(t, x) \in Q_T.
\end{equation}
Therefore, $u\in L^\infty(Q_T; \mathcal A)$.

Lastly, an entropy-entropy dissipation relation, which is formally derived in Section~\ref{sec:entrop}, holds for this system, so that the functional 
$$
E: \left\{
\begin{array}{ccc}
 L^\infty(\Omega, \mathcal A) & \to & \bR\\
 u:=(u_1,\cdots,u_n) & \mapsto & \int_\Omega \sum_{i=1}^n u_i \log u_i\\
\end{array}
\right. 
$$
is a Lyapunov function for the Stefan-Maxwell system. More precisely, it holds that 
\begin{equation}\label{eq:entcont}
\frac{d}{dt}{E}(u(t))  + \frac{\alpha}{2}  \int_\Omega \sum_{i=1}^n | \nabla \sqrt{u_i}|^2 + \frac{c^*}{2}  \int_\Omega \sum_{i=1}^n  | J_i|^2 \leq 0, 
\end{equation}
for some positive constants $\alpha, c^*>0$ whose definitions are made precise in the next section.

\subsection{Continuous entropy estimate}\label{sec:entrop}

We formally derive here the entropy-entropy dissipation inequality (\ref{eq:entcont}) 
which holds for the continuous system and was rigorously proved in~\cite{jungel2013existence}. 
For the formal calculations to hold, 
we make the simplifying assumption in this Section that the solution $u$ to the Stefan-Maxwell model satisfies 
\begin{equation}\label{eq:ass}
\forall 1\leq i \leq n, \; u_i(t,x) >0 \quad \mbox{ and }\quad \sum_{i=1}^n u_i(t,x) = 1 \quad \mbox{ a.e. in } Q_T, 
\end{equation}
and that the solution enjoys enough regularity to justify the calculations.

To present the entropy-entropy dissipation inequality which holds for the Stefan-Maxwell model, 
we need to introduce some additional notation. 
Denote by $$c^* = \min_{1\leq i\neq j \leq n} c_{ij} >0,$$ then 
for all $1\leq i\neq j\leq n$, we define 
$$
\overline{c}_{ij}:= c_{ij} - c^* \quad \mbox{ and } \quad \overline{c}:= \max_{1\leq i \neq j \leq n} \overline{c}_{ij}.
$$
Let us point out that $\overline{c}_{ij} \geq 0$ for all $1\leq i\neq j \leq n$ (and thus $\overline{c}\geq 0$).

Let ${\rm I}$ denote the $n\times n$ identity matrix. For all $v\in \bR^n$, we introduce 
$\overline{A}(v):=(\overline{A}_{ij}(v))_{1\leq i,j \leq n}$ and $C(v):=(C_{ij}(v))_{1\leq i,j \leq n}$  the matrices respectively defined as follows: for all $1\leq i,j \leq n$,
\begin{equation}\label{eq:defAbar}
\overline{A}_{ii}(v):= \sum_{1\leq j \neq i \leq n}\overline{c}_{ij}v_j, \quad \overline{A}_{ij}(v):= - \overline{c}_{ij} v_i \quad \mbox{ and } \quad C_{ij}(v):= v_i.
\end{equation}
It then holds that for all $v:=(v_1,\cdots,v_n)\in (\bR_+)^n$,
\begin{equation}\label{eq:identity}
A(v) = c^* \langle \un,v \rangle {\rm I} \clem{-} c^* C(v) + \overline{A}(v),
\end{equation}
In particular, if $u\in \bR_+^n$ satisfies $\langle \un, u\rangle = 1$, then 
\begin{equation}\label{eq:newid}
A(u) = c^*{\rm I} \clem{-} c^* C(u) + \overline{A}(u).
\end{equation}
One easily deduces from particular form~\eqref{eq:defAbar} of the matrix $\overline A(v)$ that 
\begin{equation}\label{eq:prop_Abar}
{\rm Span}\{v\} \subset  {\rm Ker}(\overline{A}(v)), \qquad 
 {\rm Ran}(\overline{A}(v)) \subset \cV_0, \qquad \forall v \in \bR^n.
\end{equation}
It has been established in~\cite{jungel2013existence} that equalities instead of mere inclusions 
hold in~\eqref{eq:prop_Abar} if one replaces $\overline A(v)$ by $A(v)$ and one considers $v$ with positive components, i.e., 
\begin{equation}\label{eq:prop_A}
{\rm Span}\{v\} = {\rm Ker}({A}(v)), \qquad 
 {\rm Ran}({A}(v)) = \cV_0, \qquad \forall v \in \left(\bR_+^*\right)^n.
\end{equation}
This property is intensively used in the convergence study of \cite{jungel2013existence}.
Provided~\eqref{eq:ass} holds, \eqref{eq:prop_A} shows that  
there exists a unique solution $J(t,x)$ to (\ref{eq:SM1comp})-(\ref{eq:SM2comp}) for almost all $(t,x)\in (0,T)\times \Omega$,  
since $\nabla u \in (\cV_0)^d$.
Besides, using (\ref{eq:newid}), it holds that $J$ is a solution to 
(\ref{eq:SM1comp})-(\ref{eq:SM2comp}) if and only if it is the unique solution to 
\begin{align}\label{eq:SM1comp-2}
& \nabla u + c^* J + \overline{A}(u) J = 0, \quad \forall 1\leq i \leq n,\\ \label{eq:SM2comp-2}
& \langle \un, J \rangle= 0,
\end{align}
since $\langle \un, u \rangle =1$ and since the condition $\langle \un, J \rangle =0$ implies that $C(u)J = 0$.

For all $v:=(v_1,\cdots,v_n) \in (\bR_+^*)^n$, we denote by $M(v):={\rm diag}(v_1, \cdots, v_n)$ the $n\times n$ diagonal matrix whose $i^{th}$ diagonal entry is given by $v_i$ for all $1\leq i \leq n$. 
Then, the following lemma, which is central in our analysis, holds.
\begin{lemma}\label{lem:M}
 Let $v:=(v_1,\cdots,v_n) \in (\bR_+^*)^n$, such that for all $1\leq i \leq n$, $v_i\leq 1$. Then, it holds that $ \overline  B(v): =M^{-1}(v) A(v)$ is a symmetric semi-definite non-negative matrix such that
 \begin{equation}\label{eq:est}
 M^{-1}(v) \overline{A}(v) \leq 2\overline{c} M^{-1}(v),
 \end{equation}
in the sense of symmetric matrices. 
\end{lemma}

\begin{proof}
  Let $v:=(v_1,\cdots,v_n) \in (\bR_+^*)^n$ and $\overline B(v):= M^{-1}(v) {\overline A}(v)$. Denoting by $\left(\overline B_{ij}(v)\right)_{1\leq i,j \leq n}$ the different components of $\overline B(v)$, a direct calculation shows that for all $1\leq i,j \leq n$, 
  $$
  \overline B_{ij}(v):= - \overline{{c}}_{ij} \mbox{ if } i \neq j \quad \mbox{ and } \quad \overline B_{ii}(v)
  =\sum_{1\leq j \neq i \leq n} \overline{ c}_{ij} \frac{u_j}{u_i},
  $$
  hence the symmetry of the matrix $\overline B(v)$. 
 Let $\xi:=(\xi_i)_{1\leq i \leq n} \in \bR^n$. Using the fact that ${\overline c_{ij} = \overline c_{ji}}$ for all $1\leq i\neq j \leq n$, it holds that
 \begin{align*}
  \xi^T \overline B(v) \xi & = \sum_{1\leq j \neq i \leq n} \overline{c}_{ij} \left(\frac{v_j}{v_i} \xi_i^2  - \xi_i \xi_j \right) \\
  & = \frac12\sum_{1\leq j \neq i \leq n} \overline{c}_{ij} \left(\frac{v_j}{v_i} \xi_i^2 + \frac{v_i}{v_j} \xi_j^2   - 2\xi_i \xi_j \right) \\
  & = \frac12\sum_{1\leq j \neq i \leq n} \overline{c}_{ij} \left(\sqrt{\frac{v_j}{v_i}} \xi_i - \sqrt{\frac{v_i}{v_j}} \xi_j\right)^2  \geq 0.
 \end{align*}
Hence the non-negativity of the matrix $\overline B(v)$. Using now {the elementary inequality 
$(a-b)^2 \leq 2 a^2 + 2 b^2$ together with} the fact that $v_i \leq 1$ for all $1\leq i \leq n$, we obtain that
\begin{align*}
   \xi^T \overline B(v) \xi & = \frac12\sum_{1\leq j \neq i \leq n} \overline{c}_{ij} 
   	\left(\sqrt{\frac{v_j}{v_i}} \xi_i - \sqrt{\frac{v_i}{v_j}} \xi_j\right)^2 \\
   & \leq \sum_{1\leq j \neq i \leq n} \overline{c}_{ij} \left(\frac{v_j}{v_i} \xi_i^2 + \frac{v_i}{v_j} \xi_j^2\right) \\
   & \leq \overline{c} \sum_{1\leq j \neq i \leq n} \left(\frac{1}{v_i} \xi_i^2 + \frac{1}{v_j} \xi_j^2\right) 
    \leq 2 \overline{c} \xi^T M^{-1}(v) \xi.
\end{align*}
Hence the desired result.
\end{proof}

We are now in position to write the (formal) entropy-entropy dissipation inequality which holds on the continuous level for the Stefan-Maxwell model. For all $1\leq i \leq n$, 
let $w_i:= D_{u_i}\mathcal E(u):= \log u_i$ and $w:=(w_i)_{1\leq i \leq n}$. Then, it holds that $\nabla u = M(u) \nabla w$ which implies that 
\begin{equation}\label{eq:w_to_J}
\nabla w = - M^{-1}(u)A(u) J = - \left( c^* M^{-1}(u) + M^{-1}(u) \overline{A}(u) \right) J.
\end{equation}
Since $M^{-1}(u)$ is {symmetric definite positive while} 
$M^{-1}(u)\overline{A}(u)$ is  symmetric non-negative, it holds that 
${c}^* M^{-1}(u) + M^{-1}(u) \overline{A}(u)$ is an invertible matrix so that $J = - \left( c^* M^{-1}(u) + M^{-1}(u) \overline{A}(u) \right)^{-1}\nabla w$. 
This yields that
\begin{equation}\label{eq:dEdt}
\frac{d}{dt}{E}(u(t))  =  
 \int_\Omega \sum_{i=1}^n  \partial_t u_i w_i \overset{\eqref{eq:masstrans}}{=}  -  \int_\Omega \sum_{i=1}^n {\rm div} J_i\, w_i
 \overset{\eqref{eq:no-flux}}{=} \int_\Omega  J \cdot \nabla w. 
 \end{equation}
Using~\eqref{eq:w_to_J}, the last term in the above equality can be rewritten of two different manners:
\begin{align}
\int_\Omega  J \cdot \nabla w = & \; - \int_\Omega  J \cdot \left( c^* M^{-1}(u) + M^{-1}(u) \overline{A}(u)\right) J 
\label{eq:J.w_1}
\\
    =& \;  - \int_\Omega  \nabla w \cdot \left( c^* M^{-1}(u) + M^{-1}(u) \overline{A}(u)\right)^{-1} \nabla w.
    \label{eq:J.w_2}
\end{align}
Define the matrix 
\begin{equation}\label{eq:B(v)}
B(v) := \left( c^* M^{-1}(v) + M^{-1}(v) \overline{A}(v)\right), \qquad \forall v = \left(v_i\right)_{1\leq i \leq n}\in (\bR_+^*)^n.
\end{equation}
It follows from Lemma~\ref{lem:M} that the two inequalities 
\begin{equation}\label{eq:B(v).ineq}
B(v) \geq c^* M^{-1}(v) \geq c^* {\rm I}, \qquad B(v)^{-1}\geq \frac{1}{c^* + 2 \overline{c}} M(v), \qquad \forall v \in (0,1]^n,
\end{equation}
hold in the sense of symmetric matrices. 
Therefore, we obtain from~\eqref{eq:J.w_1}--\eqref{eq:J.w_2} that 
\begin{equation}\label{eq:J.w_3}
\int_\Omega  J \cdot \nabla w \geq \frac1{2(c^* + 2 \overline c)} \int_\Omega  \nabla w \cdot M(u) \nabla w + 
\frac{c^\star}2 \int_\Omega |J|^2.
\end{equation}
The first term of the righthand side can be rewritten by noticing that 
$$
 \nabla w \cdot M(u) \nabla w =  \sum_{i = 1}^n u_i \nabla \log(u_i) \cdot \nabla \log(u_i) \\
 = 
 4 \sum_{i = 1}^n \left|\nabla \sqrt {u_i} \right|^2.
$$
As a consequence, we finally deduce from \eqref{eq:dEdt} and \eqref{eq:J.w_3} that 
$$
 \frac{d}{dt}{E}(u(t))  \leq  - \frac{1}{2}\alpha \int_\Omega \sum_{i=1}^n  | \nabla \sqrt{u_i}|^2- \frac{1}{2}c^*\int_\Omega  |J|^2, \\
$$
with 
$$
\alpha:= \frac{4}{c^* + 2 \overline{c}} >0.
$$
This entropy-entropy dissipation inequality is similar to (\ref{eq:entcont}). 

\begin{remark}\label{rmk:regularity}
Since the \laurent{entropy} $E$ is bounded on $L^\infty(Q_T;\cA)$ --- it takes its values in $[- m_\Omega \log(n),0]$ ---
integrating~\eqref{eq:entcont} over $t \in (0,T)$ yields 
$$
\iint_{Q_T} \left|\nabla \sqrt{u} \right|^2 + \iint_{Q_T} \left|J \right|^2 \leq C. 
$$
Moreover, since $u$ is uniformly bounded between $0$ and $1$, one has
$$ 
\iint_{Q_T} \left|\nabla \sqrt{u} \right|^2 \geq \frac14 \iint_{Q_T} \left|\nabla {u} \right|^2,
$$
so that one gets a control over the $L^2(0,T; H^1(\Omega))$ norm of $u$ and on the $L^2(Q_T)$ norm of $J$. 
This motivates the weak formulation used in Definition~\ref{def:defsol}. 
\end{remark}

\subsection{Contributions and positionning of the paper}

The goal of this paper is to build and analyze a numerical scheme preserving the properties discussed
in the previous section, namely:
\begin{itemize}
\item the non-negativity of the concentrations;
\item the conservation of mass;
\item the preservation of the volume filling constraint;
\item the entropy-entropy dissipation relation (\ref{eq:entcont}).
\end{itemize}
The scheme proposed here relies on two-point flux approximation (TPFA) finite volumes~\cite{eymard2000finite, eymard2014tp} 
and builds on similar ideas as the one introduced in~\cite{cances2020convergent} for another family of cross-diffusion systems. 

TPFA finite volumes is popular to approximate conservation laws. Unsurprisingly, schemes entering this family of methods 
have been proposed for the Stefan-Maxwell diffusion problem in~\cite{SPW03, boudin2012mathematical, PDBLGM11}. 
Those schemes yield satisfactory numerical outputs but there is no theoretical guarantee of their convergence. 
Besides, a finite element scheme is proposed and analysed in~\cite{JL19} for the more complex case where 
the chemical species are ions inducing a self-consistent electrical potential. The analysis carried out in~\cite{JL19} 
relies on the very strong assumption that integrals of non-polynomial functions can be computed exactly. 

Convergence proofs for finite volume approximations of cross-diffusion systems have been proposed 
in~\cite{ABRB11, Ahmed_intrusion, CFS_arXiv, cances2019finite, JZ19, cances2020convergent, DJZ_arXiv, murakawa2017linear, gerstenmayer2019comparison}. 
Most of the above contributions rely on the entropy-stability of the schemes, 
which is exploited thanks to the so-called discrete entropy method~\cite{CH14_FVCA7}. 
This approach is a transposition to the discrete setting of the boundedness-by-entropy method 
exposed in~\cite{jungel2015boundedness, Juengel16}. The design of entropy stable numerical schemes 
for diffusion type equations has received an important attention in the last years. Let us mention the 
contributions~\cite{BC12, BCJ14, BCV14, CG16_MCOM, CG_VAGNL, Cances_OGST, KSW18, SCS18, ABPP19, SCS19, SXY19, CCFG_HAL}, this list being non-exhaustive. 
We mention in particular the recent work~\cite{huo2020energy} where the authors propose an energy stable and positivity-preserving scheme for the Maxwell-Stefan diffusion system, but 
where no convergence analysis of the scheme is provided.  

Let us also mention that finite element methods are also used for the simulation of cross-diffusion systems. We refer the reader to~\cite{frittelli2017lumped, barrett2004finite,gurusamy2018finite} for more details. We would like to highlight in 
particular the very recent work~\cite{braukhoff2020entropy} where the authors propose a space-time Galerkin method which preserves the entropy structure of cross-diffusion systems\clem{, including the Stefan-Maxwell system under consideration}.

The scheme is presented in Section~\ref{sec:scheme}. 
Our main results are gathered in Section~\ref{sec:mainres}. 
Preliminary estimates and existence of a solution to the discretized scheme are proved in Section~\ref{sec:proofdiscrete}. 
Convergence of the discretized solution to a weak solution of the continuous model is proved in Section~\ref{sec:conv}. Finally, numerical tests illustrating the behaviour of the method are presented in Section~\ref{sec:num}.

\section{The finite-volume scheme}\label{sec:scheme}

\subsection{Discretization of $(0,T)\times \Omega$}

As already mentioned, our scheme relies on TPFA finite volumes.
As explained in~\cite{droniou2014finite,eymard2014tp,gartner2019we}, this approach appears to be very efficient as soon as the continuous problem to be solved numerically is isotropic and one has the freedom to choose a suitable mesh fulfilling the so-called 
orthogonality condition~\cite{herbin1995error,eymard2000finite}. We recall here the definition of such a mesh.

\begin{definition}\label{def:mesh}
 An admissible mesh of $\Omega$ is a triplet $(\cT, \cE, (x_K)_{K\in \cT})$ such that the following conditions are fulfilled.
 \begin{itemize}
  \item [(i)] Each control volume (or cell) $K \in \cT$ is non-empty, open, polyhedral and convex. We assume that
  $$
  K \cap L = \emptyset \mbox{ if } K,L \in \cT \mbox{ with } K \neq L, \quad \mbox{ while }\bigcup_{K\in \cT} \overline{K} = \overline{\Omega}. 
  $$
  \item[(ii)] Each face $\sigma \in \cE$ is closed and is contained in a hyperplane of $\bR^d$, with positive $(d-1)$-dimensional Hausdorff (or Lebesgue) measure denoted by $m_\sigma= \cH^{d-1}(\sigma)>0$. We assume that 
  $\cH^{d-1}(\sigma \cap \sigma') = 0$ for $\sigma ,\sigma'\in \cE$ unless $\sigma  = \sigma'$. For all $K\in \cT$, we assume that there exists a subset $\cE_K$ of $\cE$ such that $\partial K = \bigcup_{\sigma \in \cE_K} \sigma$. Moreover, we suppose that 
  $\bigcup_{K\in \cT} \cE_K = \cE$. Given two distinct control volumes $K,L\in \cT$, the intersection $\overline{K} \cap \overline{L}$ either reduces to a single face $\sigma \in \cE$ denoted by $K|L$, or its $(d-1)$-dimensional Hausdorff measure is $0$.
  \item[(iii)] The cell-centers $(x_K)_{K\in \cT}$ satisfy $x_K \in K$, and are such that, if $K,L \in \cT$ share a face $K|L$, then the vector $x_L - x_K$ is orthogonal to $K|L$.
 \end{itemize}
\end{definition}

We denote by $m_K$ the $d$-dimensional Lebesgue measure of the control volume $K$. 
The set of the faces is partitioned into two subsets: the set $\cE_{\rm int}$ of the interior faces defined by 
$$\cE_{\rm int} = \{ \sigma \in \cE \; | \; \sigma = K|L \mbox{ for some }K,L\in \cT\},$$ 
and the set $\cE_{\rm ext}=\cE \setminus \cE_{\rm int}$ of the exterior faces defined by 
$\cE_{\rm ext}=\{ \sigma \in \cE \; | \; \sigma \subset \partial \Omega \}$. 
For a given control volume $K\in \mathcal T$, we also define
$\cE_{K,{\rm int}} = \cE_K \cap \cE_{\rm int}$ (respectively $\cE_{K, {\rm ext}} = \cE_K \cap \cE_{\rm ext}$) 
the set of its faces that belong to $\cE_{\rm int}$ (respectively $\cE_{\rm ext}$). For such a face $\sigma \in \cE_{K,{\rm int}}$, we may write $\sigma = K|L$, meaning that $\sigma = \overline{K} \cap \overline{L}$, where $L\in \mathcal T$. 

Given $\sigma \in \cE$, we let 
$$
d_\sigma:= \left\{ 
\begin{array}{ll}
 |x_K - x_L| & \quad \mbox{ if } \sigma = K|L \in \cE_{\rm int},\\
 |x_K - x_\sigma| & \quad \mbox{ if } \sigma \in \cE_{K, {\rm ext}},\\
\end{array}
\right.
 \quad \mbox{ and } \quad \tau_\sigma = \frac{m_\sigma}{d_\sigma}.
$$
For internal edges $\sigma = K|L \in \cE_{\rm int}$, we also define 
$$
d_{K\sigma} = {\rm dist}(x_K, \sigma) \quad \text{and}\quad \tau_{K\sigma} = \frac{m_\sigma}{d_{K\sigma}}.
$$
Moreover, for all $K\in \cT$ and all $\sigma \in \cE_K$, we denote by
$$
n_{K\sigma}:= \left\{ 
\begin{array}{ll}
 \frac{x_L - x_K}{d_\sigma} & \quad \mbox{ if } \sigma = K|L \in \cE_{K, {\rm int}},\\
 \frac{x_\sigma - x_K}{d_\sigma} & \quad \mbox{ if } \sigma \in \cE_{K, {\rm ext}},\\
\end{array}
\right.
$$
the unitary normal to $\sigma$ outward with respect to $K$. The half-diamond cell $\Delta_{K\sigma}$ 
associated to $K$ and $\sigma$ is defined as the convex hull of $x_K$ and $\sigma$, and we define the diamond cells $\Delta_\sigma$
by 
$$
\Delta_{\sigma} = \begin{cases}
\Delta_{K\sigma} \cup \Delta_{L\sigma} & \text{if}\; \sigma = K|L \in \cE_{\rm int}, \\
\Delta_{K\sigma} &  \text{if}\; \sigma \in \cE_{K,\rm ext}.
\end{cases}
$$
Then it follows from the an elementary geometrical property that the ($d$- dimensional) Lebesgue measures 
of $\Delta_{\sigma}$ (resp. $\Delta_{K\sigma}$) are given by 
\begin{equation}\label{eq:mDelta_sig}
m_{\Delta_\sigma} = \frac{m_\sigma d_\sigma}d, \quad m_{\Delta_{K\sigma}} = \frac{m_\sigma d_{K\sigma}}d,
\end{equation}

We finally introduce the size $h_\cT$ and the regularity $\zeta_\cT$ (which is assumed to be positive) of a discretization $(\cT, \cE, (x_K)_{K\in\cT})$ of $\Omega$ by setting
$$
h_\cT = \mathop{\max}_{K \in \cT} {\rm diam}(K) \quad \mbox{ and } \zeta_\cT = \mathop{\min}_{K \in \cT} \mathop{\min}_{\sigma \in \cE_K} \frac{d(x_K, \sigma)}{d_\sigma}.
$$
Concerning the time discretization of $(0,T)$, we consider $P_T\in \bN^*$ and an increasing infinite family of times $0<t_0 < t_1 < \cdots < t_{P_T} = T$. We denote by $\Delta t_p=t_p - t_{p-1}$ for $p\in \{1, \cdots, P_T\}$, 
by $\boldsymbol{\Delta t}=(\Delta t_p)_{1\leq p \leq P_T}$, and by $h_T=\max_{1\leq p \leq P_T}\Delta t_p$. In what follows, we will use boldface notation for mesh-indexed families, typically for elements of $\bR^\cT$, $\bR^\cE$, $(\bR^\cT)^n$, $(\bR^\cE)^n$, 
$(\bR^\cT)^{P_T}$, $(\bR^\cE)^{P_T}$ or even $(\bR^\cT)^{n\times P_T}$ and $(\bR^\cE)^{n\times P_T}$. 
One naturally defines discrete $L^2$ scalar products on $\bR^\cT$ and $\bR^{d\times\cE}$ by setting
$$
\langle \bbu, \bbv \rangle_\cT = \sum_{K\in\cT} m_K u_K v_K, \qquad  
\bbu = \left(u_K\right)_{K\in\cT}, \bbv = \left(v_K\right)_{K\in\cT} \in \bR^\cT
$$
and 
$$
\langle \bbF, \bbG \rangle_\cE = \sum_{\sigma\in\cE} m_{\Delta_{\sigma}} F_{K\sigma} \cdot G_{K\sigma}, \qquad  
\bbF = \left(F_{K\sigma}\right)_{\sigma\in\cE}, \bbG = \left(G_{K\sigma}\right)_{\sigma\in\cE} \in \bR^{d\times \cE}.
$$

\subsection{Numerical scheme}

The initial data $u^0 \in L^\infty(\Omega; \mathcal A)$ is discretized into
$$
\boldsymbol{u}^0 = \left( \boldsymbol{u}_i^0\right)_{1\leq i \leq n} \in (\bR^\cT)^n = \left( u^0_{i,K}\right)_{K\in \cT, 1\leq i \leq n},
$$
by setting 
\begin{equation}\label{eq:definit}
u^0_{i,K} = \frac{1}{m_K}\int_K u_i^0(x)\,dx, \quad \forall K\in \cT, 1\leq i \leq n. 
\end{equation}
Assume that $\boldsymbol{u}^{p-1} = \left( u_{i,K}^{p-1}\right)_{K\in \cT, 1\leq i \leq n}$ is given for some $p\geq 1$, then we have to define how to compute the discrete volume fractions $\boldsymbol{u}^{p} = \left( u_{i,K}^{p}\right)_{K\in \cT, 1\leq i \leq n}$ and the 
discrete fluxes $\bbJ^p = \left(J_{i,K\sigma}^p \right)_{\sigma \in \cE, 1 \leq i \leq n}$.

First, we introduce some notation. Given any discrete scalar field $\boldsymbol{v} = (v_K)_{K\in \cT} \in \bR^\cT$, we define for all cell $K\in \cT$ and interface $\sigma \in \cE_K$ the mirror value $v_{K\sigma}$ of $v_K$ across $\sigma$ by setting:
$$
v_{K\sigma} = \left\{
\begin{array}{ll}
 v_L &  \mbox{ if } \sigma = K|L \in \cE_{\rm int},\\
 v_K & \mbox{ if } \sigma \in \cE_{\rm ext}.\\
\end{array}
\right.
$$
We also define the oriented and absolute jumps of $\boldsymbol{v}$ across any edge by
$$
D_{K\sigma} \boldsymbol{v} = v_{K\sigma} - v_K, \quad 
\mbox{ and } \quad D_\sigma \boldsymbol{v} = 
| D_{K\sigma} \boldsymbol{v} |, \quad \forall K\in \cT, \; \forall \sigma \in \cE_K.
$$
Note that in the above definition, for all $\sigma \in \cE$, the definition of 
$D_\sigma \bbv$ does not depend on the choice of the element $K\in \cT$ such that $\sigma \in \cE_K$. 

\medskip

For all $1\leq i \leq n$, we also introduce some edge values $u^p_{i,\sigma}$ of the volume fraction $u_i$ for all $\sigma \in \cE$. For any $K\in \cT$ such that $\sigma \in \cE_K$, the definition of $u^p_{i,\sigma}$ makes use of the values
$u^p_{i,K}$ and $u^p_{i,K\sigma}$ but is independent of the choice of $K$. 
As in~\cite{cances2020convergent}, the edge volume fractions $u^p_{i,\sigma}$ is defined through a logarithmic mean as follows
\begin{subequations}\label{eq:scheme}
\begin{equation}\label{eq:u_isig}
u^p_{i,\sigma} = \left\{
\begin{array}{ll}
 0 & \mbox{ if } \min(u_{i,K}^p, u_{i, K\sigma}^p)\leq 0, \\
 u_{i,K}^p & \mbox{ if } 0\leq u_{i,K}^p =u_{i, K\sigma}^p, \\
 \frac{u_{i,K}^p - u_{i,K\sigma}^p}{\log(u_{i,K}^p) - \log(u_{i,K\sigma}^p)} & \mbox{ otherwise}.\\
\end{array}
\right.
\end{equation}
We also denote by $u^p_\sigma:=\left( u^p_{i,\sigma}\right)_{1\leq i \leq n}$. This choice for the edge concentration is crucial for the preservation at the discrete level of a discrete entropy-entropy dissipation inequality similar to 
(\ref{eq:entcont}) on the continuous level. 

\medskip

The conservation laws are discretized in
a conservative way with a time discretization relying on the backward Euler scheme:
\begin{equation}\label{eq:cons}
m_K \frac{u_{i,K}^p - u_{i,K}^{p-1}}{\Delta t_p} + \sum_{\sigma \in \cE_K} m_\sigma J_{i, K\sigma}^p = 0, \quad \forall K\in \cT, \; \forall 1\leq i \leq n. 
\end{equation}
The relation between the fluxes and the variations of the volume fractions across the edges 
relies on formula~\eqref{eq:SM1comp-2} rather that on~\eqref{eq:SM1comp}. This trick takes its inspiration 
in \cite{cances2020convergent}, and appears to be crucial in what follows for the derivation of the 
discrete counterpart of the entropy-entropy dissipation estimate~\eqref{eq:entcont}. 
More precisely, the discrete fluxes $J_{K\sigma}^p:=\left(J_{i, K\sigma}^p\right)_{1\leq i \leq n}$ are solution to the following set of equations: for all $K\in \cT$ and $\sigma \in \cE_{K, {\rm int}}$, 
$$
 \frac{1}{d_\sigma} D_{K\sigma} \boldsymbol{u}_i^p + c^* J_{i,K\sigma}^p + \sum_{1\leq j \leq n} \overline{A}_{ij}(u_\sigma^p) J_{j,K\sigma}^p  = 0, \quad \forall 1\leq i \leq n,
 $$
which rewrites in a more compact form as
\begin{equation}\label{eq:SM1comp-2-disc}
 \frac{1}{d_\sigma} D_{K\sigma} \boldsymbol{u}^p + c^* J_{K\sigma}^p + \overline{A}(u_\sigma^p) J_{K\sigma}^p  = 0.
\end{equation}
One readily checks that Formula~\eqref{eq:SM1comp-2-disc} yields conservative fluxes, i.e.,
\begin{equation}\label{eq:fluxcond1}
J_{K\sigma}^p +J_{L\sigma}^p = 0, \quad \forall \sigma = K|L \in \cE_{\rm int}, \; 1 \leq p \leq P_T. 
\end{equation}
The discrete counterpart to the no-flux boundary condition~\eqref{eq:no-flux} is naturally 
\begin{equation}\label{eq:fluxcond2}
J_{K\sigma}^p = 0, \quad \forall \sigma \in \cE_{K, \rm ext}, \; K \in \cT,  \; 1 \leq p \leq P_T. 
\end{equation}
\end{subequations}

\begin{remark}\label{rmk:flux_sum_0}
We stress on the fact here that we do not impose the constraint $J_{K\sigma}^p \in \cV_0$ for all $K \in \cT, \sigma \in \cE_K$, and 
$1 \leq p \leq P_T$.
Indeed, (\ref{eq:SM1comp-2-disc}) can be rewritten equivalently as 
$$
 \frac{1}{d_\sigma} D_{K\sigma} \boldsymbol{u}^p + \left( c^*I + \overline{A}(u_\sigma^p)\right) J_{K\sigma}^p  = 0,
$$
and the matrix $c^*I + \overline{A}(u_\sigma^p)$ differs in general from $A(u_\sigma^p)$ since $u_\sigma^p$ does not belong 
to $\cV_1$ in general. 
As a consequence, ${\rm Ker}\left( c^* I +  \overline{A}(u_\sigma^p)\right)$ may not be of dimension $1$. 
Actually, we will see in Lemma~\ref{lem:first} and Lemma~\ref{lem:AT} that for any 
$\bbu^{p-1} \in \cA^\cT$, then any solution $\bbu^p$ to the scheme presented above belongs to $\cA^\cT$ and that there exists 
a unique set of fluxes $\left(J_{K\sigma}^p \right)_{K\in \cT, \sigma\in \cE_K}$ satisfying (\ref{eq:SM1comp-2-disc})-(\ref{eq:fluxcond1})-(\ref{eq:fluxcond2}), and that $J_{K\sigma}^p$ necessarily belongs to $\cV_0$.

\end{remark}

\subsection{Main results and organisation}\label{sec:mainres}

We gather the main results of our paper in this section. 
Our first theorem concerns the existence of \laurent{a} discrete solution for a given mesh, 
and the preservation of the structural properties listed in Section~\ref{sec:prop}.

\laurent{In order to obtain} a discrete counterpart of the entropy-entropy dissipation inequality~\eqref{eq:entcont}, 
we need to introduce the discrete entropy functional $E_\cT: (\bR_+^\cT)^n \to \bR$, which is defined by
\begin{equation}\label{eq:discreteent}
E_\cT(\bbv) = \sum_{i=1}^n \sum_{K\in\cT} m_K v_{i,K}\log(v_{i,K}), \quad \forall \bbv = (\bbv_i)_{1\leq i \leq n} \in (\bR_+^\cT)^n.
\end{equation}
Note that the functional $E_\cT$ is uniformly bounded on the set 
$$
\cA^\cT = \left\{ \bbv \in  (\bR_+^\cT)^n\; \middle|\; \left(v_{i,K}\right)_{1 \leq i \leq n} \in \cA \text{ for all $K\in\cT$} \right\}. 
$$
More precisely, there holds
\begin{equation}\label{eq:E.bound}
- m_\Omega \log(n) \leq  E_\cT(\bbv) \leq 0, \qquad \forall \bbv \in \cA^\cT. 
\end{equation}

Denote by $\boldsymbol{1}_\cT = (1,\dots, 1) \in \bR^\cT$, 
then the following theorem holds:
\begin{theorem}\label{th:discrete}
 Let $(\cT,\cE, (x_K)_{K\in\cT})$ be an admissible mesh and let
 $\bbu^0$ be defined by (\ref{eq:definit}) from an initial condition
 $u^0\in L^\infty(\Omega; \cA)$ \laurent{satisfying} the nondegeneracy assumption~\eqref{eq:masspos.init}.
 Then, for all $1\leq p \leq P_T$, the nonlinear system of equations (\ref{eq:scheme}) has (at least) a (strictly) 
 positive solution $\bbu^p \in \cA^\cT$. This solution $\bbu^p$ satisfies $\langle \bbu^p, \1_\cT \rangle_\cT = M$ and the corresponding fluxes 
 $\bbJ^p = \left(J_{K\sigma}^p\right)_{\sigma \in \cE}$ 
 are uniquely determined by \eqref{eq:SM1comp-2-disc}-\eqref{eq:fluxcond1}-\eqref{eq:fluxcond2} and belong to 
 $ (\cV_0)^\cE$, i.e. $\sum_{i=1}^n J_{i,K\sigma}^p = 0$ for all $\sigma \in \cE$. 
 Moreover, the following entropy-entropy dissipation estimate holds: 
 \begin{equation}\label{eq:discrete.entro}
 E_\cT (\bbu^p) +\laurent{\Delta t_p}\sum_{\sigma = K|L \in \cE_{\rm int}} \left( \frac{c^*}2 m_\sigma d_\sigma |J_{K\sigma}^p|^2 
+ \frac{\alpha}2 \tau_\sigma \left| D_{K\sigma} \sqrt{\bbu^p} \right|^2\right)\\
 \leq  E_\cT (\bbu^{p-1}).
 \end{equation}
\end{theorem}
The proof of Theorem~\ref{th:discrete} will be the purpose of Section~\ref{sec:proofdiscrete}.

\medskip

From an iterated discrete solution $(\bbu, \bbJ) = (\bbu^p, \bbJ^p)_{1\leq p \leq P_T}$ to the scheme (\ref{eq:scheme}), 
we define for all $1\leq i \leq n$, the piecewise constant approximate volume fractions 
$u_{i,\cT, \boldsymbol{\Delta t}}: Q_T \to (0,1)$ defined almost everywhere by
\begin{equation}\label{eq:approx.u}
u_{i,\cT, \boldsymbol{\Delta t}}(t,x) = u_{i,K}^n \quad \mbox{ if }(t,x)\in (t_{p-1}, t_p]\times K. 
\end{equation}
Since $\bbu^p \in \cA^\cT$, then $u_{\cT, \boldsymbol{\Delta t}} = \left(u_{i,\cT, \boldsymbol{\Delta t}}\right)_{1\leq i \leq n}$ 
belongs to $L^\infty(Q_T;\cA)$.
We also define approximate fluxes 
$J_{\cE, \boldsymbol{\Delta t}} = \left(J_{i,\cE, \boldsymbol{\Delta t}}\right)_{1\leq i \leq n}: Q_T \to (\cV_0)^d$ 
from the discrete fluxes $\bbJ^p$ by setting 
\begin{equation}\label{eq:approx.J}
J_{\cE, \boldsymbol{\Delta t}}(t,x) = d  \, J_{K\sigma}^pn_{K\sigma} \quad 
\text{if}\; (t,x) \in (t_{p-1}, t_p] \times \Delta_\sigma.
\end{equation}

\medskip
We are now in position to present our second main result, which concerns the convergence of the scheme 
as the discretisation parameters tend to $0$. 
In what follows, 
let $(\cT_m, \cE_m, (x_K)_{K\in \cT_m})_{m\geq 1}$ and $(\boldsymbol{\Delta t}_m)_{m\geq 1}$ be sequences of admissible 
discretisations of $\Omega$ and $(0,T)$ respectively. We assume that 
\begin{equation}\label{eq:h_to_0}
h_{\cT_m} \underset{m\to\infty} \longrightarrow 0, \quad h_{T_m} \underset{m\to\infty} \longrightarrow 0, 
\quad \text{while}\quad \liminf_{m \geq 1} \zeta_{\cT_m} = \zeta^* >0.
\end{equation} 
Then, the following theorem holds: 
\begin{theorem}\label{th:convergence}
Let $(\cT_m, \cE_m, (x_K)_{K\in \cT_m})_{m\geq 1}$ and $(\boldsymbol{\Delta t}_m)_{m\geq 1}$ be sequences of admissible 
discretisations of $\Omega$ and $(0,T)$ respectively fulfilling condition~\eqref{eq:h_to_0}. Let 
$\left(\bbu_m, \bbJ_m\right)_{m}= \left( \left(\bbu^p, \bbJ^p\right)_{1 \leq p \leq P_{T,m}}\right)_{m \geq 1}$ be a corresponding 
sequence of discrete solutions to \eqref{eq:scheme}, from which a sequence of approximate solutions 
$\left(u_{\cT_m, \boldsymbol{\Delta t}_m}, J_{\cE_m, \boldsymbol{\Delta t}_m}\right)_{m\geq 1}$ is 
reconstructed thanks to~\eqref{eq:approx.u}--\eqref{eq:approx.J}.
Then there exists a weak solution $(u,J)$ to \eqref{eq:masstrans}-\eqref{eq:SM1comp}-\eqref{eq:SM2comp} 
in the sense of Definition~\ref{def:defsol} such that, up to a upsequence,
$$
u_{\cT_m, \boldsymbol{\Delta t}_m} \mathop{\longrightarrow}_{m\to +\infty} u \mbox{ a.e. in } Q_T,
$$
and 
$$
J_{\cE_m, \boldsymbol{\Delta t}_m} \mathop{\rightharpoonup}_{m\to +\infty} J \mbox{ weakly in }L^2((0,T)\times \Omega)^{d \times n}.
$$
\end{theorem}
The proof of Theorem~\ref{th:convergence} is the purpose of Section~\ref{sec:conv}. It is based on compactness 
arguments that are deduced from the {\it a priori} estimates established in Theorem~\ref{th:discrete}.

\section{
Numerical analysis at fixed grid
}\label{sec:proofdiscrete}

This section is devoted to the proof of Theorem~\ref{th:discrete}. 

\subsection{A priori estimates}\label{sec:estimates}

The first lemma shows the non-negativity and the mass conservation of the solution to (\ref{eq:scheme}), together with the uniqueness of associated fluxes.

\begin{lemma}\label{lem:first}
 Given $\boldsymbol{u}^{p-1} \in \cA^\cT$ satisfying  
 \begin{equation}\label{eq:mconv0}
 \langle \1_\cT, \bbu^{p-1} \rangle_\cT = M \in \left(\bR_+^*\right)^n,
 \end{equation}
then any solution $\boldsymbol{u}^p$ to (\ref{eq:scheme}) satisfies 
$\langle \1_\cT, \bbu^{p} \rangle_\cT = M$
 and is positive in the sense that
 $u_{i,K}^p >0$ for all $K\in\cT$ and all $1\leq i \leq n.$
Besides, for any solution $\boldsymbol{u}^p$ to  (\ref{eq:scheme}), there exists a unique set of fluxes $\bbJ^p$ 
satisfying (\ref{eq:SM1comp-2-disc})-(\ref{eq:fluxcond1})-(\ref{eq:fluxcond2}). 
\end{lemma}

\begin{proof}
 Let $\boldsymbol{u}^p$ be a solution to (\ref{eq:scheme}) 
 and let $1\leq i \leq n$. Let us first prove that the total volume of each specie is conserved, so that $\langle \1_\cT, \bbu^{p} \rangle_\cT = M$.
Summing equation~\eqref{eq:cons} over $K\in\cT$ gives
$$ 
\langle \1_\cT, \bbu^{p} \rangle_\cT - \langle \1_\cT, \bbu^{p-1} \rangle_\cT 
= - \laurent{\Delta t_p}\sum_{\sigma = K|L \in \cE_{\rm int}}m_\sigma \left(J_{K\sigma}^p + J_{L\sigma}^p \right) 
- \laurent{\Delta t_p}\sum_{\sigma \in \cE_{\rm ext}} m_\sigma J_{K\sigma}^p. 
$$
Then it follows directly from the local conservativity of the scheme~\eqref{eq:fluxcond1} and from the 
discrete no-flux boundary condtion~\eqref{eq:fluxcond2} that 
$$\langle \1_\cT, \bbu^{p} \rangle_\cT = \langle \1_\cT, \bbu^{p-1} \rangle_\cT = M.$$
\medskip
 
Let us now prove that $\boldsymbol{u}^p$ is positive. Let $1\leq i \leq n$. We consider a cell $K\in \cT$ where $\boldsymbol{u}_i^p$ reaches its minimum, i.e. such that $u_{i,K}^p \leq u_{i,L}^p$ for all $L\in \cT$, and denote 
 $w_i^p:= u_{i,K}^p  = \min_{L\in \cT} u_{i,L}^p$. \laurent{Assume} for contradiction that $w_i^p = u_{i,K}^p \leq 0$. Let us recall again equation (\ref{eq:cons}), which implies that
 \begin{equation}\label{eq:cons2}
 m_K \frac{u_{i,K}^p - u_{i,K}^{p-1}}{\Delta t_p} = - \sum_{\sigma \in \cE_K} m_\sigma J_{i, K\sigma}^p.
 \end{equation}
 On the one hand, the term on the left-hand side is non-positive since $u_{i,K}^{p-1} \geq 0 \geq u_{i,K}^p$. 
On the other hand, the specific choice~\eqref{eq:u_isig} for the edge volume fractions implies 
that $u_{i,\sigma}^p = 0$ for all $\sigma \in \cE_K$.
Therefore, $\overline{A}_{ij}(u_\sigma^p) = 0$ for all $1\leq j \neq i \leq n$. As a consequence, 
relation~\eqref{eq:SM1comp-2-disc} reduces to
$$
\frac1{d_\sigma} D_{K\sigma}\boldsymbol{u}_i^p + \left( c^* +\sum_{1\leq j \neq i \leq n} \overline{c}_{ij} u_{j,\sigma}^p\right) J_{i, K\sigma}^p = 0.
$$
Since $u_{j,\sigma}^p \geq 0$, $\overline{c}_{ij} \geq 0$ and  $D_{K\sigma}\boldsymbol{u}_i^p \geq 0$ 
for all $1\leq i,j \leq n$, we obtain that $J_{i, K\sigma}^p \leq 0$ for all $\sigma \in \cE_K$. Using 
(\ref{eq:cons2}),  this yields that 
$J_{i, K\sigma}^p = \tau_\sigma D_{K\sigma}\boldsymbol{u}_i^p= 0$ for all $\sigma \in \cE_K$. As a consequence, 
$u_{i,K}^p = u_{i,L}^p$ for all $L\in \cT$ such that $\sigma = K|L \in \cE_K$. Iterating this argument and since $\Omega$ is connected, 
we thus obtain that $u_{i,L}^p = w_i^p \leq 0$ for all $L\in \cT$. 
This implies that $\langle \bbu_i^p, \1_\cT \rangle_\cT \leq 0$ which yields a contradiction with the property 
$\langle \1_\cT, \bbu_i^{p} \rangle_\cT =M_i >0$ we just
established. Thus, $\boldsymbol{u}^p$ is positive. 

\medskip

As a consequence, for all $\sigma\in\cE_{{\rm int}}$ and all $1\leq i \leq n$, $u^p_{i,\sigma}>0$. The fact that 
there exists a unique $\bbJ^p$ associated to $\bbu^p$ via (\ref{eq:SM1comp-2-disc})-(\ref{eq:fluxcond1})-(\ref{eq:fluxcond2}) 
is then a consequence of Lemma~\ref{lem:M}.
Indeed, for all 
$K\in \cT$ and all $\sigma \in \cE_{K,{\rm int}}$, noticing that $D_{K\sigma} \bbu^p = M(u_\sigma) D_{K\sigma}\log(\bbu^p)$,
we can rewrite equivalently (\ref{eq:SM1comp-2-disc}) as 
$$
\frac1{d_\sigma} M(u^p_\sigma)  D_{K\sigma}\log(\bbu^p) + \left(c^* {\rm I} + \overline{A}(u^p_\sigma)\right)J_{K\sigma}^p = 0.
$$
The positivity of $u^p_\sigma$ implies the inversibility of matrix $M(u^p_\sigma)$. As a consequence, it holds that
\begin{equation}\label{eq:eqinter}
\frac1{d_\sigma} D_{K\sigma}\log(\bbu^p) + \left(c^*  M(u^p_\sigma)^{-1}  + 
M(u^p_\sigma)^{-1} \overline{A}(u^p_\sigma)\right)J_{K\sigma}^p = 0.
\end{equation}
Moreover, thanks to Lemma~\ref{lem:M}, it holds that $B(u_\sigma^p) = c^*  M(u^p_\sigma)^{-1}  + M(u^p_\sigma)^{-1} \overline{A}(u^p_\sigma)$ is a symmetric positive definite matrix, and the only solution $J_{K\sigma}^p$ to (\ref{eq:eqinter}) is given by
\begin{equation}\label{eq:J_Dlog}
J_{K\sigma}^p = - \frac1{d_\sigma} B(u_\sigma^p)^{-1} D_{K\sigma}\log(\bbu^p).
\end{equation}
Hence the desired result. 
\end{proof}

The next lemma shows that the total discrete flux vanishes across all edges and 
that the volume filling constraint is automatically satisfied without being enforced. 

\begin{lemma}\label{lem:AT}
Given $\boldsymbol{u}^{p-1}\in \mathcal A^{\mathcal T}$ satisfying \eqref{eq:mconv0}, 
any solution $(\bbu^p, \bbJ^p)$ to \eqref{eq:scheme} belongs to $\cA^\cT \times (\cV_0)^\cE$. 
\end{lemma}

\begin{proof}
Since $\bbu^{p-1}\in \mathcal A^{\mathcal T}$ satisfies (\ref{eq:mconv0}), $\bbu^{p-1}$ is nonnegative, 
and using Lemma~\ref{lem:first}, any corresponding solution $\bbu^p$ to 
(\ref{eq:scheme}) is then positive. 

Let us denote by $\bbw^{p-1} = (w_K^{p-1})_{K\in \cT}:= \langle \un, \bbu^p \rangle$, and let us denote by  
$G_{\sigma K}^p:= 
\langle \un, J_{K\sigma}^p \rangle$ for all $K\in\cT$ and $\sigma \in \cE_K$.
Summing equations (\ref{eq:cons}) for $i=1,\cdots,n$, we obtain that 
$$
 m_K \frac{w_{K}^p - w_{K}^{p-1}}{\Delta t_p} = - \sum_{\sigma \in \cE_K} m_\sigma G_{K\sigma}^p.
$$
In addition, summing~\eqref{eq:SM1comp-2-disc} over $i$ provides that for all $\sigma = K|L \in \cE_{\rm int}$, 
$$
 = \frac1{d_\sigma} D_{K\sigma} \boldsymbol{w}^p + c^* G_{K\sigma}^p + 
 \left \langle \un, \overline{A}(u_\sigma^p) J_{K\sigma}^p \right\rangle\\
 \overset{\eqref{eq:prop_Abar}}{=} \frac1{d_\sigma} D_{K\sigma} \boldsymbol{w}^p + c^* G_{K\sigma}^p.
$$
Thus, $\bbw$ is solution to the classical backward Euler TPFA scheme for the heat equation with diffusion coefficient $\frac{1}{c^*}$. This scheme is well-posed and 
$\bbw^p = \bbw^{p-1} = \bbun_\cT$ is its unique solution, which implies that $\bbu \in \cA^\cT$. 
Moreover, the fluxes $G_{K\sigma}^p$ are all equal to zero, so that $J_{K\sigma} \in (\cV_0)^\cE$.
\end{proof}

The last statement of this section is devoted to the entropy entropy-dissipation estimate~\eqref{eq:discrete.entro}.
\begin{lemma}\label{lem:entropy}
Given $\boldsymbol{u}^{p-1} \in \mathcal A^{\mathcal T}$, any solution 
$\left(\boldsymbol{u}^p, \bbJ^p\right) \in \mathcal A^{\mathcal T} \times (\cV_0)^\cE$ to \eqref{eq:scheme} satisfies
 \begin{equation}\label{eq:discent}
 E_\cT(\boldsymbol{u}^{p}) + 
 \Delta t_p \sum_{\sigma = K|L \in \cE_{\rm int}} \left( \frac{c^*}2 m_\sigma d_\sigma |J_{K\sigma}^p|^2 
+ \frac{\alpha}2 \tau_\sigma \left| D_{K\sigma} \sqrt{\bbu^p} \right|^2\right) \leq E_\cT(\boldsymbol{u}^{p-1}).
 \end{equation} 
\end{lemma}

\begin{proof}
 Multiplying equation (\ref{eq:cons}) by $\Delta t_p \log(u_{i,K}^p)$ (which makes sense since $\bbu^p$ is positive 
 owing to Lemma~\ref{lem:first}), and summing over all the cells and species leads to
 \begin{equation}\label{eq:T1+T2}
 T_1 + T_2 = 0, 
 \end{equation}
 where we have set 
 \begin{align*}
 T_1 = &  \sum_{K\in \cT} \sum_{i=1}^n \left[ u_{i,K}^p \log(u_{i,K}^p) - u_{i,K}^{p-1}\log(u_{i,K}^p)\right]m_K, \\
 T_2 = & \Delta t_p\sum_{i=1}^n\sum_{K\in\cT}\sum_{\sigma \in \cE_K} m_\sigma J_{i, K\sigma}^p\log(u_{i,K}^p).
 \end{align*}
 On the one hand, using the convexity of the function $\bR_+ \ni x \mapsto x\log x$, it holds that
$$
u_{i,K}^p - u_{i,K}^{p-1} + u_{i,K}^p \log(u_{i,K}^p) - u_{i,K}^{p-1}\log(u_{i,K}^p) \geq u_{i,K}^p \log(u_{i,K}^p) - u_{i,K}^{p-1}\log(u_{i,K}^{p-1}), 
$$
which implies, together with Lemma~\ref{lem:AT}, that
\begin{equation}\label{eq:T1}
T_1 \geq  E_{\cT}(\boldsymbol{u}^p) - E_\cT(\boldsymbol{u}^{p-1}). 
\end{equation}
On the other hand, the conservativity of the fluxes~\eqref{eq:fluxcond1} and the discrete no-flux boundary 
condition~\eqref{eq:fluxcond2} allow to reorganise the term $T_2$ as
$$
T_2 = - \Delta t_p \sum_{\sigma = K|L \in \cE_{\rm int}} m_\sigma \left\langle J_{K\sigma}^p, D_{K\sigma} \log(\bbu^p) \right\rangle.
$$
Bearing in mind the expression~\eqref{eq:J_Dlog} of the fluxes, 
\begin{align*}
- \left\langle J_{K\sigma}^p, D_{K\sigma} \log(\bbu^p) \right\rangle = & 
\frac{d_\sigma} 2 \left\langle J_{K\sigma}^p, B(u_\sigma^p) J_{K\sigma}^p \right\rangle \\
&+ \frac1{2 d_\sigma}\left\langle D_{K\sigma} \log(\bbu^p), B(u_\sigma^p)^{-1} D_{K\sigma} \log(\bbu^p) \right\rangle. 
\end{align*}
Then estimates~\eqref{eq:B(v).ineq} provide that 
$$
 \left\langle J_{K\sigma}^p, B(u_\sigma^p) J_{K\sigma}^p \right\rangle \geq c^* \left|J_{K\sigma}\right|^2
$$
and 
$$
\left\langle D_{K\sigma} \log(\bbu^p), B(u_\sigma^p)^{-1} D_{K\sigma} \log(\bbu^p) \right\rangle 
\geq \frac{\alpha}4 \left\langle D_{K\sigma} \log(\bbu^p), M(u_\sigma^p) D_{K\sigma} \log(\bbu^p) \right\rangle. 
$$
Thanks to the particular choice~\eqref{eq:u_isig} for $u_\sigma^p$, the right-hand side rewrites
$$
 \left\langle D_{K\sigma} \log(\bbu^p), M(u_\sigma^p) D_{K\sigma} \log(\bbu^p) \right\rangle 
= 
 \left\langle D_{K\sigma} \log(\bbu^p), D_{K\sigma} \bbu^p\right\rangle \\
\geq  
4 \left| D_{K\sigma} \sqrt{\bbu^p} \right|^2, 
$$
the last inequality being a consequence of the elementary inequality 
$$(a-b)(\log(a) - \log(b))\geq 4 (\sqrt a - \sqrt b)^2$$ 
holding for any positive $a,b$.
Summing up, we have 
\begin{equation}\label{eq:T2}
T_2 \geq \Delta t_p \sum_{\sigma = K|L \in \cE_{\rm int}} \left( \frac{c^*}2 m_\sigma d_\sigma |J_{K\sigma}^p|^2 
+ \frac{\alpha}2 \tau_\sigma \left| D_{K\sigma} \sqrt{\bbu^p} \right|^2\right).
\end{equation}
To conclude the proof, it only remains to incorporate~\eqref{eq:T1} and \eqref{eq:T2} in \eqref{eq:T1+T2}.
\end{proof}

\subsection{Existence of discrete solutions}\label{sec:exis}

The purpose of this section is to prove the existence of a solution to \eqref{eq:scheme}.

\begin{proposition}
 Given $\bbu^{p-1} \in \cA^\cT$ satisfying~\eqref{eq:mconv0}, then there exists at least one solution 
 $\left(\bbu^p, \bbJ^p\right) \in \cA^\cT \times (\cV_0)^\cE$ to the scheme~\eqref{eq:scheme}.
\end{proposition}

\begin{proof}
 The proof relies on a topological degree argument~\cite{leray1934topologie,deimling2010nonlinear}. 
 The idea is to transform continuously our complex nonlinear system into a linear system while 
 guaranteeing that enough {\it a priori} estimates controlling the solution remain valid all along the homotopy. 
 We sketch the main ideas of the proof, making the homotopy explicit. 
 
 For $\lambda \in [0,1]$, we look for $\left(\bbu^{(\lambda)}, \bbJ^{(\lambda)}\right) \in \bR^{n\times \cT} \times \bR^{n\times\cE}$ 
 solution to the algebraic system~\eqref{eq:scheme} where the matrix $\overline A(u_{\sigma}^p)$ is replaced 
 by $\lambda \overline A(u_{\sigma}^{(\lambda)})$.
Our system~\eqref{eq:scheme} corresponds to the case $\lambda = 1$, whereas the case
$\lambda = 0$ corresponds to the usual TPFA finite volume scheme $n$ decoupled heat 
equations all with the same diffusion coefficient $\frac1{c^*}$.
Mimicking the calculations presented in Section~\ref{sec:estimates}, one shows that whatever $\lambda \in [0,1]$, 
any corresponding solution $\left(\bbu^{(\lambda)}, \bbJ^{(\lambda)}\right)$ lies in $\cA^\cT \times (\cV_0)^\cE$, 
and $\bbu^{(\lambda)}$ is positive. Moreover, the entropy - entropy dissipation estimate and the 
uniform bound~\eqref{eq:E.bound} on the entropy ensure that 
$$
\left\| \bbJ\laurent{^{(\lambda)}}\right\|_\cE^2 \leq \frac{2m_\Omega \log
  n}{c^*\laurent{\Delta t_p}} =: {K}.
$$
where 
$
\left\| \bbJ\laurent{^{(\lambda)}}\right\|_\cE^2 =  \sum_{\sigma=K|L\in\cE_{\rm int}} m_\sigma d_\sigma |J_{K\sigma}\laurent{^{(\lambda)}}|^2.
$
Fixing $\eta>0$, we define the relatively compact open sets
$$
\cA_\eta^\cT =\left\{ \bbu \in (\bR^\cT)^n \;\middle|\; \mathop{\inf}_{\bbv \in \cA^\cT} \|\bbu - \bbv\| < \eta \right\}
$$
and 
$$
(\cV_0)^\cE_{\eta} =\left\{ \bbJ \in (\bR^\cE)^n \;\middle|\; 
\left\| \bbJ\right\|_\cE^2
< K^{1/2} + \eta \quad \text{and}\quad
 \mathop{\inf}_{\bbF \in (\cV_0)^\cE} \|\bbJ - \bbF\| < \eta \right\}.
$$
The {\it a priori} estimates ensure that no solution $\left(\bbu^{(\lambda)}, \bbJ^{(\lambda)}\right)$ of the modified scheme 
can cross the boundary of the open set $\cA_\eta^\cT \times (\cV_0)^\cE_{\eta}$. The topological degree associated to the 
modified scheme and $\cA_\eta^\cT \times (\cV_0)^\cE_{\eta}$ is constant with respect to $\lambda$, and takes the value $+1$ 
for $\lambda = 0$ since the system is linear and invertible with positive determinant. So it is also equal to $1$ for $\lambda = 1$, 
ensuring the existence of a solution to the nonlinear problem~\eqref{eq:scheme}.
\end{proof}

The proof of Theorem~\ref{th:discrete} is now complete.

\section{Proof of Theorem~\ref{th:convergence}}\label{sec:conv}

We consider here a sequence $\left( \cT_m, \cE_m, (x_K)_{K\in \cT_m}\right)_{m\geq 1}$ of admissible space discretizations with $h_{\cT_m}$ going to $0$ as $m$ tends to $+\infty$, 
while the regularity $\zeta_{\cT_m}$ remains uniformly bounded from below by a positive constant $\zeta^*$.
We also consider a sequence  $(\boldsymbol{\Delta t}_m)_{m\geq 1}=\left((\Delta t_{p,m})_{1\leq p \leq P_{T,m}}\right)_{m\geq 1}$ of admissible time discretizations such that $h_{T,m}$ goes to $0$ as $m$ goes to infinity.

From the discrete solutions $\left( \bbu_m , \bbJ_m\right)$, $m\geq 1$, 
the existence of which being guaranteed by Theorem~\ref{th:discrete}, we 
reconstruct the piecewise constant functions $u_{\cT_m, \boldsymbol{\Delta t}_m} \in L^\infty(Q_T;\cA)$ and 
$J_{\cE_m, \boldsymbol{\Delta t}_m} \in L^2(Q_T;\cV_0)^d$ thanks to formulas~\eqref{eq:approx.u} and \eqref{eq:approx.J}.
In the convergence analysis, we also need the weakly consistent piecewise constant gradient reconstruction operators
$\nabla_{\cE_m}$ and $\nabla_{\cE_m,\boldsymbol{\Delta t}_m}$ defined for $m \geq 1$ and  $\bbv \in \bR^{\cT_m}$
\begin{equation}\label{eq:nabla_E.1}
\nabla_{\cE_m} \bbv (x) = d D_{K\sigma} \bbv_m n_{K\sigma}\quad \text{if}\; x \in \Delta_\sigma,\; \sigma \in \cE_m, 
\end{equation}
and, for   $\bbv  = \left(\bbv^p\right)_{0\leq p \leq P_{T,m}} \in \bR^{(1+P_{T,m}) \times \cT_m}$, 
\begin{equation}\label{eq:nabla_E.2}
\nabla_{\cE_m, \boldsymbol{\Delta t}_m} \bbv_m (t,\cdot) = \nabla_{\cE_m} \bbv^p \quad \text{if}\; t\in (t_{p-1},t_p], 
\; 1\leq p \leq P_{T,m}.
\end{equation}

\subsection{Compactness on approximate reconstructions}

The next proposition is the main result of this section.

\begin{prop}\label{prop:compact}
There exists $u\in L^\infty(Q_T;\cA^\cT) \cap L^2(0,T;H^1(\Omega))^n$ with $\sqrt{u} \in L^2(0,T;H^1(\Omega))^n$, 
and $J \in L^2(Q_T; (\cV_0)^d)$ such that, up to a subsequence, the following convergence properties hold:
\begin{align}\label{eq:compact.u}
u_{\cT_m, \boldsymbol{\Delta t}_m} \underset{m\to+\infty}\longrightarrow & u \quad \text{a.e. in } Q_T, 
\\
\label{eq:compact.gradsqrtu}
\nabla_{\cE_m, \boldsymbol{\Delta t}_m} \sqrt{\bbu_m}  \underset{m\to+\infty}\longrightarrow & \nabla \sqrt{u} 
\quad \text{weakly in } L^2(Q_T)^{n\times d}, 
\\
\label{eq:compact.gradu}
\nabla_{\cE_m, \boldsymbol{\Delta t}_m} {\bbu_m}  \underset{m\to+\infty}\longrightarrow & \nabla {u} 
\quad \text{weakly in } L^2(Q_T)^{n\times d}, 
\\
\label{eq:compact.J}	
J_{\cE_m, \boldsymbol{\Delta t}_m}  \underset{m\to+\infty}\longrightarrow & J
\quad \text{weakly in } L^2(Q_T)^{n\times d}.
\end{align}
\end{prop}
\begin{proof}
Summing~\eqref{eq:discrete.entro} over $p  \in \{1,\dots, P_{T,m}\}$ and using the bound~\eqref{eq:E.bound} on $E_\cT$ 
provides
\begin{equation}\label{eq:compact.1}
\sum_{p=1}^{P_{T,m}} \Delta t_p \sum_{\sigma \in \cE_{{\rm int},m}} \left(
\frac{\alpha}2 \tau_\sigma \left| D_\sigma \sqrt{\bbu_m} \right|^2 + \frac{c^*}2 m_\sigma d_\sigma \left| J_{K\sigma} \right|^2 \right)
\leq m_\Omega \log n.
\end{equation}
Recalling the elementary geometrical relation $d m_{\Delta_\sigma} = m_\sigma d_\sigma$ and the definitions~\eqref{eq:approx.J} 
of $J_{\cE_m, \boldsymbol{\Delta t}_m}$ and~\eqref{eq:nabla_E.1}-\eqref{eq:nabla_E.2}, one obtains that 
\begin{equation}\label{eq:compact.2}
\left\| J_{\cE_m, \boldsymbol{\Delta t}_m} \right\|_{L^2(Q_T)^{n\times d}}  + \left\| 
\nabla_{\cE_m, \boldsymbol{\Delta t}_m} \sqrt{\bbu_m} \right\|_{L^2(Q_T)^{n\times d}} \leq C
\end{equation}
for some $C$ not depending on $m$.
As a straightforward consequence, there exists $J, F \in L^2(Q_T)^{n\times d}$ such that~\eqref{eq:compact.J} holds, as well as 
\begin{equation}\label{eq:compact.3}
\nabla_{\cE_m, \boldsymbol{\Delta t}_m} \sqrt{\bbu_m}  \underset{m\to+\infty}\longrightarrow  F
\quad \text{weakly in } L^2(Q_T)^{n\times d}.
\end{equation}
The fact that $J \in  L^2(Q_T; \cV_0)^d$ results from the stability of linear space $\cV_0$ for the 
weak convergence. 
Moreover, since $0\leq u_K^n \leq 1$, then $D_{\sigma}\bbu_m^p \leq 2 D_{\sigma}\sqrt{\bbu_m^p}$ 
for all $\sigma \in \cE_{{\rm int},m}$ and all $1 \leq p \leq P_{T,m}$.
Therefore, we deduce from~\eqref{eq:compact.2} that 
$$
\left\| \nabla_{\cE_m, \boldsymbol{\Delta t}_m} {\bbu_m} \right\|_{L^2(Q_T)^{n\times d}} \leq C, 
$$
whence the existence of some $G \in L^2(Q_T)$ such that 
\begin{equation}\label{eq:compact.4}
\nabla_{\cE_m, \boldsymbol{\Delta t}_m} {\bbu_m}  \underset{m\to+\infty}\longrightarrow  G
\quad \text{weakly in } L^2(Q_T)^{n\times d}.
\end{equation}
On the other hand, $u_{\cT_m,  \boldsymbol{\Delta t}_m}$ belongs to the bounded subset 
$L^\infty(Q_T;\cA)$ of $L^\infty(Q_T)^n$ for all $m \geq 1$. Therefore, up to a subsequence, 
 $u_{\cT_m,  \boldsymbol{\Delta t}_m}$ converges in the $L^\infty(Q_T)^n$-weak star sense 
 towards some $u$, which takes its values in $\cA$ since both the positivity and the sum to 1 property 
 are stable when passing to the limit in this topology. 
 
To conclude this proof, it remains to check that the convergence of  $u_{\cT_m,  \boldsymbol{\Delta t}_m}$ 
towards $u$ holds point-wise, and to identify $F$ and $G$ as $\nabla \sqrt{u}$ and $\nabla u$ respectively. 
These properties are provided all at once by the nonlinear discrete Aubin-Simon lemma \cite[Theorem 3.9]{ACM17}. 
As already established in~\cite{ACM17}, this theorem applies naturally in the TPFA finite volume context. The only 
point to be checked is a discrete $L^2(0,T;H^{-1}(\Omega))$ estimate on the time increments of 
$u_{\cT_m,  \boldsymbol{\Delta t}_m}$.
More precisely, for $\phi \in C^\infty_c((0,T) \times \Omega;\bR^n)$, one defines 
$\bbphi = \left( \phi_{i,K}^p \right) \in \bR^{n\times P_{T,m} \times \cT_m}$ by 
$$
\phi_{i,K}^p = \frac{1}{\Delta t_p \, m_K} \int_{t_{p-1}}^{t_p} \int_K \phi_i(t,x) {\rm d}x {\rm d}t.
$$
It follows from 
\eqref{eq:cons}-\eqref{eq:fluxcond1}-\eqref{eq:fluxcond2} that 
$$
\sum_{p = 1}^{P_{T,m}} \sum_{K\in\cT_m} m_K \langle (u_K^p - u_{K}^{p-1}), \phi_K^p\rangle
= \sum_{p = 1}^{P_{T,m}} \Delta t_p \sum_{\sigma \in \cE_{{\rm int}, m}} m_\sigma \left \langle J_{K\sigma}^p, 
D_{K\sigma} \bbphi^p \right \rangle.
$$
Applying Cauchy-Schwarz inequality leads to 
\begin{multline*}
\sum_{p = 1}^{P_{T,m}} \sum_{K\in\cT_m} m_K \langle (u_K^p - u_{K}^{p-1}), \phi_K^p\rangle\\
\leq \left(\sum_{p = 1}^{P_{T,m}} \Delta t_p \sum_{\sigma \in \cE_{{\rm int}, m}} m_\sigma d_\sigma \left|J_{K\sigma}^p\right|^2 \right)^{1/2}
\left(\sum_{p = 1}^{P_{T,m}} \Delta t_p \sum_{\sigma \in \cE_{{\rm int}, m}} \tau_\sigma \left| D_\sigma \bbphi^p \right|^2 \right)^{1/2}.
\end{multline*}
The discrete $L^2(Q_T)^d$ estimate on the fluxes~\eqref{eq:compact.1} shows that the first term in 
the righthand side is bounded, whereas the second term is the discrete $L^2(0,T;H^1(\Omega))$ semi-norm of $\bbphi$.
A straightforward generalisation of \cite[Lemma 9.4]{eymard2000finite} shows that 
$$
\sum_{p = 1}^{P_{T,m}} \Delta t_p \sum_{\sigma \in \cE_{{\rm int}, m}} \tau_\sigma \left| D_\sigma \bbphi^p \right|^2
\leq C \|\nabla\phi\|_{L^2(Q_T)^d}^2
$$
for some $C$ only depending on the regularity factor $\zeta^*$.
Therefore, 
$$
\sum_{p = 1}^{P_{T,m}} \sum_{K\in\cT_m} m_K \langle (u_K^p - u_{K}^{p-1}), \phi_K^p\rangle 
\leq C \|\nabla\phi\|_{L^2(Q_T)^d}\leq C  \|\nabla\phi\|_{L^\infty(Q_T)^d}, 
$$
which is exactly the condition required to apply \cite[Theorem 3.9]{ACM17}, which provides~\eqref{eq:compact.u}-\eqref{eq:compact.gradsqrtu}-\eqref{eq:compact.gradu} all at once, concluding the proof of Proposition~\ref{prop:compact}.
\end{proof}

For all $m\geq 1$, we introduce the diamond cell based reconstruction $u_{\cE_m, \boldsymbol{\Delta t}_m}$ 
of the volume fractions defined by 
$$
u_{\cE_m, \boldsymbol{\Delta t}_m} (t,x) = u_\sigma^p \quad \text{if}\; (t,x) \in (t_{p-1}, t_p] \times \Delta_\sigma, 
\sigma \in \cE_{m},\; 1 \leq p \leq P_{T,m}, 
$$
where the $u_\sigma^p$ are given by~\eqref{eq:u_isig}. The following lemma shows that both reconstructions 
$u_{\cE_m, \boldsymbol{\Delta t}_m}$ and $u_{\cT_m, \boldsymbol{\Delta t}_m}$ share the same limit $u$. 
The proof is omitted there since it is similar to the one of~\cite[Lemma 4.4]{cances2020convergent}.
\begin{lem}\label{lem:conv_usig}
Let $u$ be as in Proposition~\ref{prop:compact} then, up to a subsequence,  
$u_{\cE_m, \boldsymbol{\Delta t}_m}$ converges in $L^r(Q_T)$, $1 \leq r < +\infty$ towards $u$ as $m$ tends to $+\infty$.
\end{lem}

\subsection{Convergence towards a weak solution}\label{ssec:identify}

Our last statement to conclude the proof of Theorem~\ref{th:convergence} consists in identifying the limit values 
$(u,J)$ of the approximate solutions as weak solutions to the Stefan-Maxwell cross-diffusion system.

\begin{proposition}\label{prop:identify}
 Let $(u,J)$ be as in Proposition~\ref{prop:compact} then $(u,J)$ is a weak solution to (\ref{eq:masstrans})-(\ref{eq:SM1comp})-(\ref{eq:SM2comp}) in the sense of Definition~\ref{def:defsol}. 
\end{proposition}
\begin{proof}
One has already established in Proposition~\ref{prop:compact} that the limit values $(u,J)$ lie in the right functional spaces. 
It only remains to check that~\eqref{eq:masstrans}, \eqref{eq:no-flux} and \eqref{eq:SM1comp-2} hold in the distributional sense. 

Equation~\eqref{eq:SM1comp-2-disc} implies that 
\begin{equation}\label{eq:identify.1}
\nabla_{\cE_m, \boldsymbol{\Delta t}_m} \bbu_m + \left(c^* I + \overline A(u_{\cE_m, \boldsymbol{\Delta t}_m})\right) 
J_{\cE_m, \boldsymbol{\Delta t}_m} = 0, \qquad \forall m \geq 1.
\end{equation}
Since $v \mapsto \overline A(v)$ is continuous, it follows from Lemma~\ref{lem:conv_usig} that 
$\overline A(u_{\cE_m, \boldsymbol{\Delta t}_m})$ tends to $A(u)$ in $L^2(Q_T)^{n\times n}$. Then thanks to the 
convergence properties~\eqref{eq:compact.gradu}-\eqref{eq:compact.J}, one can pass to the weak limit in~\eqref{eq:identify.1} 
to recover that~\eqref{eq:SM1comp-2} holds in $L^1(Q_T)^{n\times d}$, thus also in $L^2(Q_T)^{n\times d}$.

Concerning equations \eqref{eq:masstrans} and \eqref{eq:no-flux}, we establish them in the distributional sense~\eqref{eq:weak.u}. 
Let $\phi \in C^\infty_c([0,T) \times \overline \Omega)$, then for $m\geq 1$, define 
$\bbphi_m = \left(\phi_{K}^p\right)_{K\in\cT_m, 1\leq p \leq P_{T,m}}$ by setting $\phi_{K}^p = \phi(t_p, x_K)$.
Multiplying~\eqref{eq:masstrans} by $\Delta t_p \phi_{K}^{p-1}$ for some $1 \leq i \leq N$ and summing 
over $K\in\cT_m$ and $1 \leq p \leq P_{T,m}$ gives after reorganisation that
\begin{multline}\label{eq:R1+R2+R3}
 \iint_{Q_T} u_{i,\cT_m, \boldsymbol{\Delta t}_m} \partial_t \phi + \int_{\Omega} u_i^0 \phi(0,\cdot) 
 +  \iint_{Q_T} J_{i,\cE_m, \boldsymbol{\Delta t}_m} \cdot \nabla \phi \\
 = R_{1,m}(\phi) + R_{2,m}(\phi) + R_{3,m}(\phi), 
\end{multline}
where we have set 
\begin{align*}
R_{1,m}(\phi) = & \sum_{p=1}^{P_{T,m}} \sum_{K\in\cT_m} m_K u_{i,K}^p \left(\phi_K^p- \phi_K^{p-1} - \frac1{m_K} \int_{t_{p-1}}^{t_p} \int_K \partial_t \phi \right),
\\
R_{2,m}(\phi) = & \sum_{K\in\cT_m} m_K u_{i,K}^0 \left( \phi_K^0 - \frac1{m_K} \int_K \phi(0,\cdot)\right),
\\
R_{3,m}(\phi) = &\sum_{p=1}^{P_{T,m}} \Delta t_p  \sum_{\sigma \in\cE_m} m_\sigma d_\sigma J_{i,K\sigma}^p 
\left( \frac1 {d_{\sigma}}D_{K\sigma} \bbphi_m^{p-1} - \frac{1}{m_{\Delta_\sigma}\Delta t_p}\int_{t_{p-1}}^{t_p}
\nabla \phi \cdot n_{K\sigma}\right).
\end{align*}
It follows from the regularity of $\phi$ that 
$$
\left|\phi_K^p- \phi_K^{p-1} - \frac1{m_K} \int_{t_{p-1}}^{t_p} \int_K \partial_t \phi \right| \leq C \Delta t_p 
(h_{\cT_m} + h_{T_m}),
$$
so that, using that $0 \leq u_{i,K}^p \leq 1$, we obtain that 
\begin{equation}\label{eq:R1}
\left| R_{1,m}(\phi) \right| \leq C (h_{\cT_m} + h_{T_m}) \underset{m\to+\infty} \longrightarrow 0. 
\end{equation} 
Similarly, one shows that 
\begin{equation}\label{eq:R2}
\left| R_{2,m}(\phi) \right| \leq C h_{\cT_m}  \underset{m\to+\infty} \longrightarrow 0. 
\end{equation} 
Finally, the orthogonality condition on the mesh, namely point (iii) of Definition~\eqref{def:mesh}, 
ensures that 
$$
\left| \frac1 {d_{\sigma}}D_{K\sigma} \bbphi_m^{p-1} - \frac{1}{m_{\Delta_\sigma}\Delta t_p}\int_{t_{p-1}}^{t_p}
\nabla \phi \cdot n_{K\sigma}\right| \leq C(h_{\cT_m} + h_{T_m}).
$$
Therefore, 
$$
\left|R_{3,m}(\phi)\right| \leq C(h_{\cT_m} + h_{T_m}) 
\left\| J_{i,\cE_m, \boldsymbol{\Delta t}_m} \right\|_{L^1(Q_T)^d} \underset{m\to+\infty} \longrightarrow 0
$$
since $\left\| J_{i,\cE_m, \boldsymbol{\Delta t}_m}
\right\|_{L^1(Q_T)^d}$ can be controlled thanks to the Cauchy-Schwarz inequality 
by 
$T^{1/2} m_\Omega^{1/2} \left\| J_{i,\cE_m, \boldsymbol{\Delta t}_m} \right\|_{L^2(Q_T)^d}$ which is bounded 
thanks to~\eqref{eq:compact.2}. 
Then in view of the convergence in $L^1(Q_T)$ of $u_{i,\cT_m,  \boldsymbol{\Delta t}_m}$ towards $u_i$ and 
of the weak convergence in $L^2(Q_T)^d$ of $J_{i,\cE_m, \boldsymbol{\Delta t}_m}$ towards $J_i$, one can 
pass to the limit in~\eqref{eq:R1+R2+R3} to recover that 
$$
 \iint_{Q_T} u_{i} \partial_t \phi + \int_{\Omega} u_i^0 \phi(0,\cdot) 
 +  \iint_{Q_T} J_{i} \cdot \nabla \phi 
 = 0.
$$
The weak formulation~\eqref{eq:weak.u} is then recovered by summing over $i$.
\end{proof}

\section{Numerical results}\label{sec:num}

The aim of this section is to collect some numerical results obtained with the numerical scheme presented in the preceding sections. The numerical scheme has been implemented using Julia and the different codes used to produce the numerical tests presented below 
can be found at~\cite{virginie_ehrlacher_2020_3934286}
(10.5281/zenodo.3934286) . The nonlinear system is solved thanks to a
modified Newton algorithm with stopping criterion $\|\bbu^{p,k+1} -
\bbu^{p,k}\|_{\ell^\infty} < 10^{-12}$ where the superscript $k$
refers to the iteration of the Newton method. 
The obtained solution, denoted by $\bbu^{p -2/3}$ is then projected onto $\mathcal A$ by setting: 
$$
\bbu^{p -1/3}= \max(\bbu^{p -2/3}, 10^{-12}) \mbox{ then } u_{i,K}^p = \frac{u_{i,K}^{p-1/3}}{\sum_{i=1}^n u_{i,K}^{p-1/3}}.
$$

\subsection{Convergence under grid refinement}

We first present some numerical results obtained on a one-dimensional test case, in order to illustrate the rate of convergence of the method with respect to the spatial discretization parameter. 
Here, $\Omega = (0,1)$, and we consider a system composed of three different species ($n=3$).
\clem{Two different initial conditions $u_0$ are considered:
\begin{itemize}
\item a smooth initial profile defined for $x\in(0,1)$ by
\begin{equation}\label{eq:smooth}
u_1^0(x) =u_2^0(x) =\frac14+\frac14\cos(\pi x);
\end{equation}
\item a non-smooth initial profile defined for $x\in(0,1)$ by
\begin{equation}\label{eq:nonsmooth}
u_1^0(x) = \un_{[3/8, 5/8]}(x), \quad u_2^0(x)  = \un_{(1/8, 3/8)}(x) + \un_{(5/8, 7/8)}(x),
\end{equation}
\end{itemize}
where $\un_E$ denotes the characteristic function of the set $E\subset [0,1]$, and 
where $u_3^0$ is deduced from $u_1^0$ and $u_2^0$ by the relation $u_3^0 = 1-u_1^0 - u_2^0$.}
The time step is chosen to be constant and equal to $\Delta t = 10^{-5}$ and final time as $T =0.5$. 
The spatial mesh is chosen to be a uniform grid of the interval $(0,1)$ containing $N$ subintervals. 

The value of the cross-diffusion coefficients are chosen to be
\[
c_{12} = c_{21} = 0.2, \; c_{13} = c_{31} = 1.0, \; c_{23} = c_{32} = 0.1, \; c^* = 0.1.
\]

Figure~\ref{fig:fig1} illustrates the evolution of the $L^1$ in time and space error of the approximate discrete solution as a function of $N$ (which is computed in comparison with an approximate solution computed on a very fine grid \clem{with $N_{\rm ref}=10^4$ cells}).  

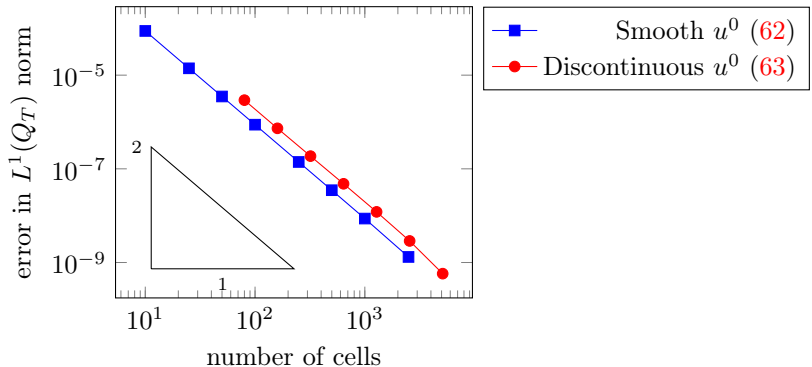
\begin{figure}[h!]
\begin{tikzpicture}
	\begin{loglogaxis}[
	xlabel=number of cells,
	ylabel=error in $L^1(Q_T)$ norm,
	legend style={
		cells={anchor=east},
		legend pos=outer north east,
	},
	width=0.5\linewidth]
	
	\addplot[mark=square*, color=blue] table[x=N, y=errSmooth] {error_new.dat};
	\addlegendentry{Smooth $u^0$~\eqref{eq:smooth}}
	\addplot[mark=*, color=red] table[x=N, y=errNonSmooth] {disc_zero.dat};
	\addlegendentry{Discontinuous $u^0$~\eqref{eq:nonsmooth}}

	\logLogSlopeTriangle{0.1}{-0.4}{0.1}{2}{black};
	\end{loglogaxis}
\end{tikzpicture}
 \caption{Evolution of the $L^1$ space time error of the approximate solution as a function of the spatial discretization parameter. }
\label{fig:fig1}
\end{figure}

We numerically observe that the error decays like $\mathcal O\left( \frac{1}{N^2}\right)$, in other words, 
\clem{showing that the scheme is second order accurate in space}.

\subsection{Two-dimensional test case}

We present here a two-dimensional test case. The number of species is kept to be $n=3$ and the values of the cross-diffusion coefficients are now given by 
\begin{equation}\label{eq:coeffs2D}
c_{12} = c_{21} = 0.1, \; c_{13} = c_{31} = 0.2, \; c_{23} = c_{32} = 2, \; c^* = 0.1.
\end{equation}
The spatial domain $\Omega = (0,1)^2$ 
is discretized using a cartesian uniform grid containing 
$70$ cells in each direction. Time step is chosen to be $\Delta t = 10^{-5}$.

Figure~\ref{fig:fig2} (respectively Figure~\ref{fig:fig3} and Figure~\ref{fig:fig4}) shows the values of the concentration profiles $u_1,u_2,u_3$ at time $t=0$ (respectively $t_1 = 8.5 \; 10^{-5}$ and $t_2 = 1 \; 10^{-3}$). 
\clem{
Since the coefficients $c_{12}$ and $c_{13}$ are much smaller than $c_{23}$, the initial interfaces 
between the different species are easily diffused for early times. Recall that $c_{ij}$ is an inverse diffusion coefficient. 
On Figure~\ref{fig:fig4}, one clearly sees that the species 2 and 3 have difficulties to interdiffuse 
due to the high value of $c_{23}$, so that 
the specie 2 remains essentially confined in a region where $u_3$ is small. 
}
\begin{figure}[htb]
 \includegraphics*[width=5cm]{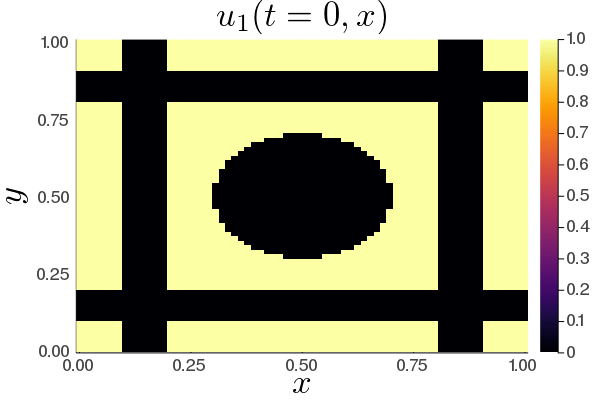}
  \includegraphics*[width=5cm]{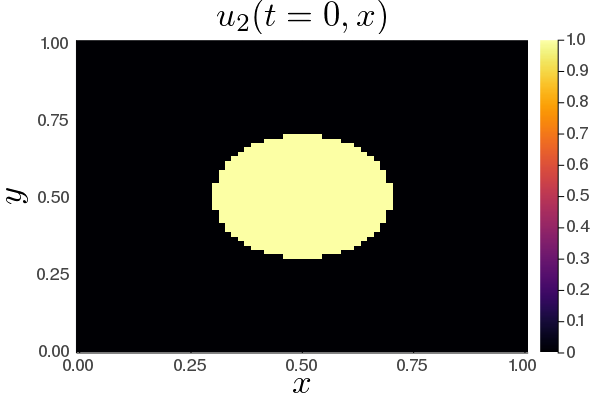}
   \includegraphics*[width=5cm]{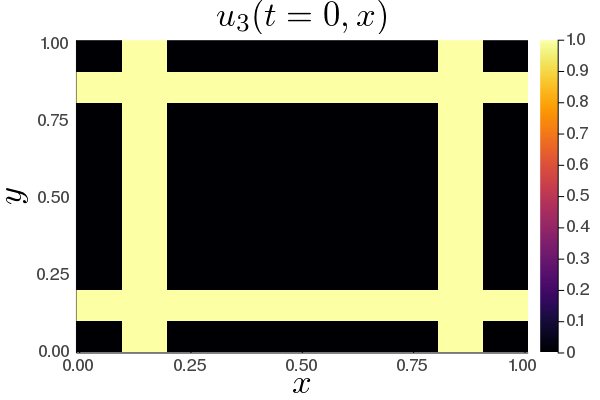}
 \caption{Initial profiles of the volume fractions.}
\label{fig:fig2}
\end{figure}

\begin{figure}[htb]
 \includegraphics*[width=5cm]{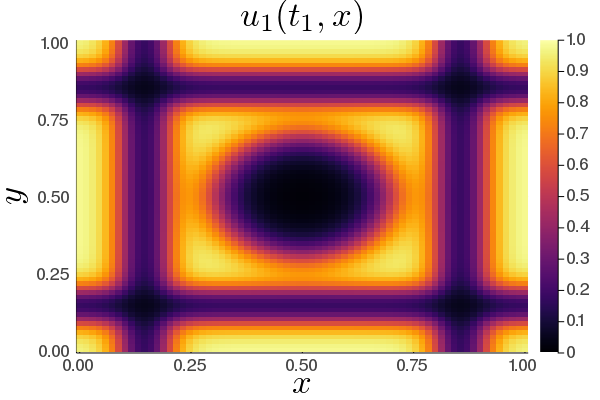}
  \includegraphics*[width=5cm]{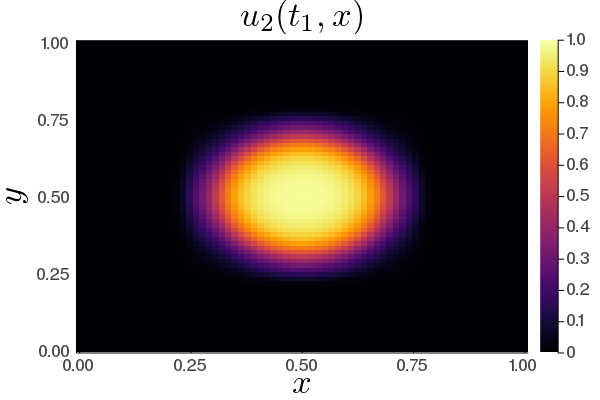}
   \includegraphics*[width=5cm]{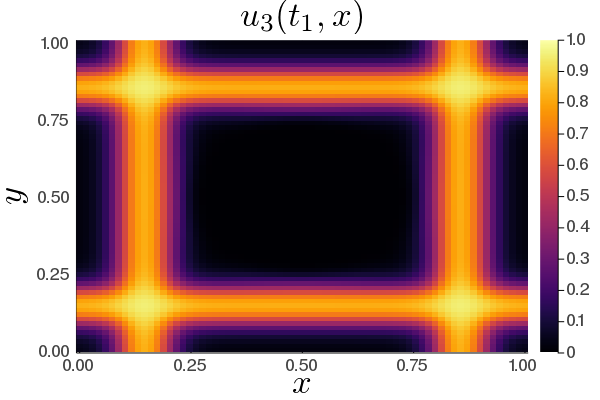}
 \caption{Profiles of the volume fractions at $t_1 = 8.5 \; 10^{-5}$.}
\label{fig:fig3}
\end{figure}

\begin{figure}[htb]
 \includegraphics*[width=5cm]{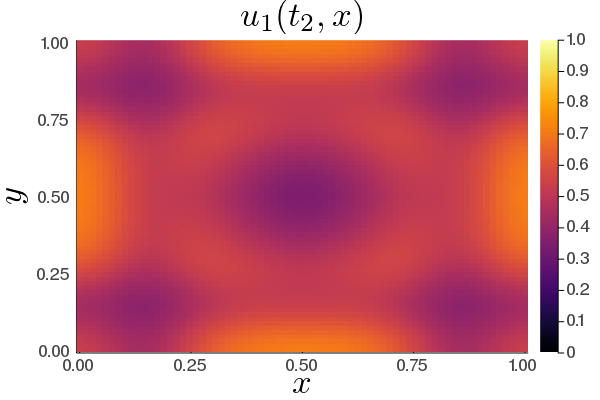}
  \includegraphics*[width=5cm]{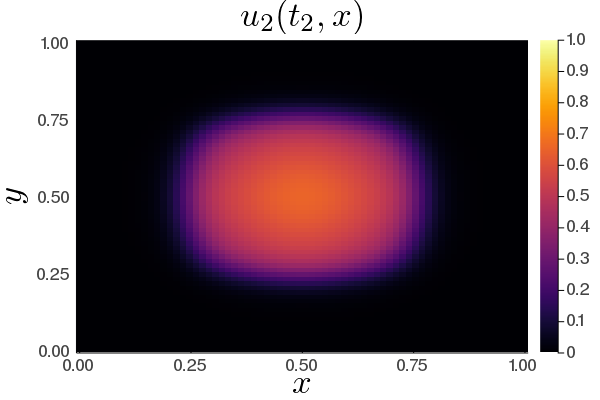}
   \includegraphics*[width=5cm]{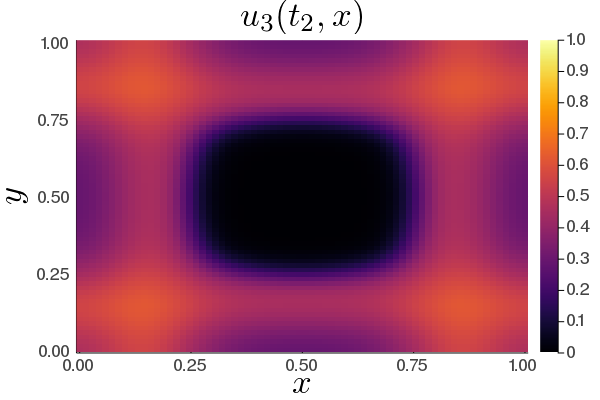}
 \caption{Profiles of the volume fractions at $t_2 = 1 \; 10^{-3}$.}
\label{fig:fig4}
\end{figure}

\clem{Our last figure is there to highlight both the decay of the discrete entropy 
and the exponential convergence towards equilibrium 
of the approximate solution. The exponential convergence in the continuous case was 
established in~\cite{jungel2013existence} thanks to a Logarithmic Sobolev inequality. 
A discrete counterpart of this inequality has been proved in~\cite{BCJ14}, allowing to 
show the exponential convergence of the approximate solution towards the constant in space 
equilibrium following the lines of~\cite{jungel2013existence}. We omit the proof here 
and rather provide a numerical evidence. 
}

\clem{
Define $\bbm:= {(m_{i,K})}_{i,K}\in (\mathbb{R}^\cT)^n$ by 
\[
m_{i,K} = \frac{1}{|\Omega|}M_i, \quad \forall 1\leq i \leq n, \quad \forall K\in \mathcal T, 
\]
and $M_i$ is defined by \eqref{eq:masspos.init}, and by 
\[
H_\cT(\bbu^p|\bbm) = \sum_{K\in\cT} \sum_{i=1}^n m_K u_{i,K}^p \log\left(\frac{u_{i,K}^p}{m_{i,K}}\right)
 = E_\cT(\bbu^p) - E_{\cT}(\bbm) \geq 0
\]
the relative entropy between the approximate solution $\bbu^p$ at the $p^\text{th}$ time step and 
the long-time limit of $u$. Figure~\ref{fig:fig5} shows that our approximate solution converges 
exponentially fast towards the right long-time limit. The exponential convergence in $L^1$ can 
then be deduced from a Csisz\'ar-Kullback inequality. 
}

\begin{figure}[htb]
\begin{tikzpicture}
	\begin{semilogyaxis}[xlabel=$t$,ylabel=relative entropy, width=0.5\linewidth]
	\addplot[mark=none, color = red] table[x=time, y=EntRel] {entropy2D.dat};
	\end{semilogyaxis}
\end{tikzpicture}
 \caption{Evolution of the relative entropy $H_{\mathcal T}(\bbu^p|\bbm)$ as a function of time. }
\label{fig:fig5}
\end{figure}
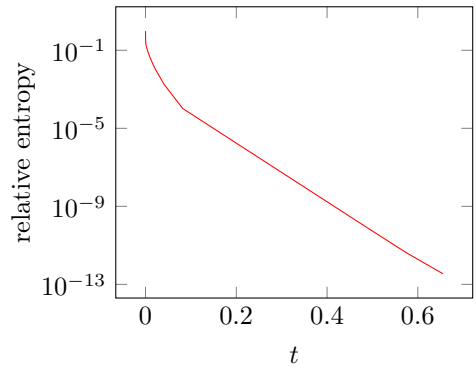

\subsection*{Acknowledgements}

The authors acknowledge support from project COMODO (ANR-19-CE46-0002). CC also acknowledges support from Labex CEMPI (ANR-11-LABX-0007-01).

\bibliographystyle{plain}
\bibliography{biblio}

\end{document}